\documentclass[11pt]{amsart}
\usepackage{amssymb,amsmath,amsthm,amscd}
\usepackage[all]{xy}

\usepackage[hypertex]{hyperref}

\numberwithin{equation}{section}

\newtheoremstyle{fancy1}{10pt}{10pt}{\itshape}{12pt}{\textsc\bgroup}{.\egroup}{8pt}{
}
\newtheoremstyle{fancy2}{10pt}{10pt}{}{12pt}{\itshape}{.}{8pt}{ }

\theoremstyle{fancy1}

\newtheorem{lem}[equation]{Lemma}
\newtheorem{prop}[equation]{Proposition}
\newtheorem{thm}[equation]{Theorem}

\newtheorem{main}{Theorem}
\newtheorem*{main*}{Theorem}
\newtheorem*{conjecture}{Conjecture}
\newtheorem*{cor*}{Corollary}

\theoremstyle{fancy2}
\newtheorem{definition}[equation]{Definition}
\newtheorem{rem}[equation]{Remark}
\newtheorem*{rem*}{Remark}
\newtheorem{example}[equation]{Example}

\newcommand{\cref}[1]{Corollary~\ref{#1}}

\newcommand{\lref}[1]{Lemma~\ref{#1}}

\newcommand{\tref}[1]{Theorem~\ref{#1}}

\newcommand{\Sph}{\mathbb{S}}
\newcommand{\Disc}{\mathbb{D}}

\newcommand{\Rho}{\mathsf{R}}
\newcommand{\Q}{\mathsf{Q}}

\newcommand{\N}{\mathsf{N}}
\newcommand{\e}{\epsilon}

\newcommand{\RP}{\mathbb{R\mkern1mu P}}
\newcommand{\CP}{\mathbb{C\mkern1mu P}}
\newcommand{\HP}{\mathbb{H\mkern1mu P}}

\newcommand{\R}{{\mathbb{R}}}
\newcommand{\Z}{{\mathbb{Z}}}

\renewcommand{\H}{\ensuremath{\operatorname{\mathsf{H}}}}

\newcommand{\F}{\ensuremath{\operatorname{\mathsf{F}}}}
\newcommand{\G}{\ensuremath{\operatorname{\mathsf{G}}}}
\newcommand{\D}{\ensuremath{\operatorname{\mathsf{D}}}}
\newcommand{\SO}{\ensuremath{\operatorname{\mathsf{SO}}}}
\renewcommand{\O}{\ensuremath{\operatorname{\mathsf{O}}}}
\newcommand{\Sp}{\ensuremath{\operatorname{\mathsf{Sp}}}}
\newcommand{\U}{\ensuremath{\operatorname{\mathsf{U}}}}
\newcommand{\SU}{\ensuremath{\operatorname{\mathsf{SU}}}}

\newcommand{\M}{\ensuremath{\operatorname{\mathsf{M}}}}

\newcommand{\T}{\ensuremath{\operatorname{\mathsf{T}}}}
\renewcommand{\S}{\ensuremath{\operatorname{\mathsf{S}}}}
\renewcommand{\r}{\ensuremath{\operatorname{\mathsf{r}}}}

\newcommand{\fg}{{\mathfrak{g}}}
\newcommand{\fk}{{\mathfrak{k}}}

\newcommand{\fa}{{\mathfrak{a}}}

\def\con#1=#2(#3){#1 \equiv #2 \bmod{#3}}

\newcommand{\Aut}{\ensuremath{\operatorname{Aut}}}
\newcommand{\Int}{\ensuremath{\operatorname{Int}}}

\renewcommand{\Im}{\ensuremath{\operatorname{Im}}}

\DeclareMathOperator{\Iso}{Iso}

\DeclareMathOperator{\Fix}{Fix}

\DeclareMathOperator{\Id}{Id}

\newcommand{\Kpm}{\mathsf{K}^{\scriptscriptstyle{\pm}}}
\newcommand{\Kp}{\mathsf{K}^{\scriptscriptstyle{+}}}
\newcommand{\Km}{\mathsf{K}^{\scriptscriptstyle{-}}}
\newcommand{\Ko}{\mathsf{K}_{\scriptscriptstyle{0}}}
\newcommand{\K}{\mathsf{K}}
\renewcommand{\L}{\mathsf{L}}

\renewcommand{\F}{\mathsf{F}}

\begin{document}

\title{Polar Manifolds and Actions}

\dedicatory{Dedicated to Richard Palais on his 80th birthday}

\author{Karsten Grove}
\address{University of Notre Dame\\
       Notre Dame, IN 46556}
\email{kgrove2@nd.edu}

\author{Wolfgang Ziller}
\address{University of Pennsylvania\\
       Philadelphia, PA 19104}
\email{wziller@math.upenn.edu}

\thanks{Both authors were supported by  a grant from the
National Science Foundation. The second author was also supported by
CAPES and IMPA in Rio de Janeiro}

\maketitle

%-------------- Article Text--------------------

Many geometrically interesting objects arise from spaces supporting
 large transformation groups. For example symmetric spaces form a
rich and important class of spaces. More generally special
 geometries such as, e.g., Einstein manifolds, manifolds of positive curvature, etc., are often
  first found and possibly classified in the homogeneous case \cite{Be,Wa,AW,Bb,Wan,BWZ,La}. Similarly, interesting
   minimal submanifolds have been found and constructed as invariant submanifolds relative to a
   large isometry group (see , e.g., \cite{HHS} and \cite{HL}).

The ``almost homogeneous" case, i.e., manifolds  with an isometric
$\G$ action with one dimensional
 orbit space, has also received a lot of attention recently. They are
 special in that one can reconstruct the manifold from the knowledge
 of its isotropy groups. This played a crucial role in the construction of a large class of new
 manifolds with nonnegative curvature in \cite{GZ1, GZ2}, in the classification problem for
 positively curved manifolds of cohomogeneity one in \cite{V1,V2, GWZ, VZ}, as well as in the
  construction of a new manifold of positive curvature \cite{De, GVZ}.

\smallskip

Our main objective in this paper is to exhibit and analyze $\G$
manifolds $M$, that like
 manifolds with orbit space an interval, can be reconstructed from its isotropy groups.
 Although such manifolds are rather rigid compared with all $\G$ manifolds, there are several
 natural constructions of such manifolds, and they form a rich and interesting class, including
 manifolds related to Coxeter groups and Tits buildings as observed and used in \cite{FGT}.

\smallskip

The special actions of interest to us here are the so-called
\emph{polar actions}, i.e., proper isometric actions for which there
is an (immersed) submanifold
 $\sigma : \Sigma \to M$, a so-called \emph{section},
that meets all orbits orthogonally, or equivalently  \emph{the
horizontal distribution} (on the regular part) \emph{is integrable}. This concept was pioneered by Szenthe in
\cite{Sz1,Sz2} and independently
 by Palais and Terng in  \cite{PT1}.  In fact
 Palais and Terng observed that this is the natural class of group
 actions where the classical Chevalley restriction theorem for
 the adjoint action of a compact Lie group
 holds, i.e., a $\G$ invariant function on $M$ is smooth
 if and only if its restriction to a section is smooth (and invariant under the stabilizer of $\Sigma$). This is also
 a
 natural class of group actions where a reduction to a potentially
 simpler lower dimensional problem along a smooth section is
 possible since in general the smooth structure of the quotient
 $M/\G$ is too complicated to do analysis effectively. This has already
 been used
 successfully in the cohomogeneity one case where solutions to a PDE
 were obtained via a reduction to an ODE, see, e.g., \cite{Bom,Co,Cs,CGLP}   where
 manifolds with special holonomy, Einstein metrics, Sasakian Einstein metrics and soliton
 metrics were produced.

\smallskip

   The most basic examples of polar actions are the
adjoint action of a compact Lie group on itself, or on its Lie
algebra. More generally, the isotropy representation of a symmetric
space $M$, either on $M$, or on $T_pM$, is polar. In fact Dadok
showed that a  linear representation which is polar is (up to orbit
equivalence) the isotropy representation of a symmetric space
\cite{Da}.  Polar actions on symmetric spaces (isometric with
respect to the symmetric metric) have been studied extensively, see,
e.g., \cite{HPTT1, HPTT2,HLO}. They were classified for  compact
rank one symmetric spaces in \cite{PTh} and for compact irreducible
symmetric spaces of higher rank it was shown \cite{Ko2,Ko3,KL,L}
 that a polar action must be hyperpolar, i.e., the section is flat,
and the hyperpolar actions were classified in \cite{Ko1}. For
non-compact symmetric actions the classification is still wide open,
see, e.g., \cite{BD,BDT}.

\smallskip

To describe the data involved in the reconstruction of a polar $\G$
manifold $M$, let $\sigma: \Sigma \to M$ be a section and $\Pi$ its
 stabilizer group (mod kernel), which we refer to as the \emph{polar group}. Then
 $\Sigma$ is a totally geodesic immersed submanifold of $M$ and $\Pi$ a discrete group of
 isometries acting properly discontinuously on $\Sigma$.
  Moreover $M/\G = \Sigma/\Pi$. The polar group contains a canonical
   general type \emph{``reflection group"} $\Rho \triangleleft \Pi$ with ``fundamental domain" a
    \emph{chamber}, $C \subset \Sigma$, locally the union of convex subsets of $\Sigma$
    (in the atypical case where $\Rho$ is trivial $C$ is all of $\Sigma$) .
    Furthermore, we have the subgroup of the polar group  $\Pi_C\subset\Pi$ which stabilizes
    $C$ with $\Pi=\Rho\cdot\Pi_C$ and $C/\Pi_C=M/\G$.

 The most important polar actions are those for which there are no exceptional orbits and $\Pi_C$ is
   trivial. We will refer to them as  {\it
    Coxeter polar actions}. For such actions, $C$ is in particular a convex set isometric to $M/\G$.

\smallskip

     The general data  needed for the reconstruction of a Coxeter polar action are as follows:
      The isotropy groups along $C$, one for each component of the orbit
     types. They
     satisfy
       \emph{compatibility relations} coming from the
     slice
     representations, which themselves are polar. They can be organized in a \emph{graph of groups}
      denoted by $\G(C)$.
      This group graph generalizes the
      concept of a group
       diagram for cohomogeneity one manifolds  \cite{GZ1,AA1}. The general \emph{polar data}
      $D$ needed are $D = (C, \G(C))$.
The following is  our  main result:

\begin{main} A Coxeter polar action $(M,\G)$ is determined by its polar data $D = (C,\G(C))$.
Specifically, there is a canonical construction of a Coxeter
polar $\G$ manifold $M(D)$ from polar data $D = (C,\G(C)$, and
if $D$ is the polar data for a Coxeter polar
 $\G$ manifold $M$ then $M(D)$ is equivariantly diffeomorphic to
 $M$.
\end{main}

For a general polar manifold we construct a $\G$-equivariant  cover
which  is Coxeter polar. Conversely, if $M$ is Coxeter polar with
data $(C,\G(C))$ and $\Gamma$ acts isometrically on $C$ preserving
the group graph, it extends to a free action on $M$ such that
$M/\Gamma$ is polar. See \tref{cover} and \tref{quotient}.

The special class of polar manifolds, where no singular orbits are
present is in a  sense ``diametrically opposite" to that of Coxeter
 polar manifolds. Those are simply described as
 \begin{center}
 $M = \G/\H \times_{\Pi}\Sigma$
\end{center}
where $\H$ is a subgroup of $\G$ and $\Pi$ a discrete subgroup of
$\N(\H)/\H$ acting properly discontinuously on $\Sigma$. We call
these actions \emph{exceptional polar} since
 this is precisely the class of polar
manifolds where the orbits are principal or exceptional, i.e., of
the same dimension, thus defining a foliation on $M$ (see
\tref{exceptional}).  These
examples also
  show that \emph{any discrete subgroup of a Lie group can be the polar
  group of a polar action}, which is our motivation for renaming these
  groups  polar groups, as opposed to generalized Weyl groups, as
  is common in the literature.

In  Theorem A one can include metrics on the section, as was shown
in \cite{Me}: If $\Sigma$ is equipped with a metric invariant under
$\Pi$, then $M(D)$ has a polar metric with section $\Sigma$,
extending the metric on $\Sigma$. This is in fact crucial in the
proof of Theorem A, as well in the construction of a large class of
examples.

 \smallskip

It is apparent from  our general result that  Coxeter polar
manifolds play a pivotal role in the theory. They are also in
various other ways the most interesting ones. For example, one
naturally
  associates to such a manifold a general \emph{chamber system} \cite{Ti2, Ro}, which provides a natural
 link to the theory of buildings. This was used in \cite{FGT} to show that \emph{polar actions on simply connected positively curved manifolds of cohomogeneity at least two are smoothly equivalent to polar actions on rank one symmetric spaces} (here passing from topological to smooth equivalence appeals to our recognition result above).  In another direction, the construction of principal
  $\L$ bundles over a cohomogeneity one manifold with orbit space an interval, crucial for the
   construction of principal bundles with total space of nonnegative curvature in \cite{GZ1,GZ2}, carries over
   to general Coxeter polar manifolds as a consequence of Theorem
   A.

\begin{main}
Let $D = (C, \G(C))$ be  data for a Coxeter polar $\G$ manifold $M$,
and let $\L$ be a Lie group.
 Then for any $\G(C)$ compatible choice of homomorphisms from each $\K \in \G(C)$ into $\L$, the resulting group
 graph $\L\times\G(C)$ together with $C$ are Coxeter data, $D' = (C, \L\times\G(C))$ for a Coxeter
  polar $\L\times\G$ manifold $P$. Moreover, $P$ is a principal $\L$ bundle over $M$.
\end{main}

 Examples are
polar actions on $\CP^n$ and $\HP^n$ which lift to polar actions on
spheres under the Hopf fibration. See also \cite{Mu} for examples of
a similar type. This construction is important as well for the
special cohomogeneity two case of $C_3$ geometry  in \cite{FGT}
where the general theory for buildings brakes down.

\medskip
Extending the concept of a \emph{hyperpolar} manifold, where  the
sections are flat (euclidean polar might be a better terminology),
we say that a polar manifold $M$ is \emph{spherical polar} resp.
\emph{hyperbolic polar} if its sections have constant positive resp.
negative sectional curvature.

Using Morse theory we provide a connection between geometric
ellipticity and topological ellipticity. Recall that a topological
space  is called  rationally elliptic if all but finitely many
homotopy groups are finite.  Rationally elliptic manifolds have many
very restrictive properties, see \cite{FHT}.
 The following
can be viewed as an extension of the main result for cohomogeneity
one manifolds in \cite{GH2}:

\begin{main}
A simply connected compact spherical polar or hyperpolar manifold
$M$ is rationally elliptic.
\end{main}

We actually prove the stronger property that the Betti numbers of
the loop space grow at most polynomially for any field of
coefficients (instead of just rational coefficients).   It is
interesting to compare Theorem C with the (extended) Bott conjecture
\cite{GH1,Gr} which
states that a manifold with (almost) nonnegative curvature is
rationally elliptic. It is thus also a natural question if one can
construct metrics with nonnegative curvature on spherical polar or
hyperpolar manifolds, possibly under special conditions on the
codimensions of the orbits, as was done in the cohomogeneity one
case in \cite{GZ1}.

\smallskip

Our main recognition theorem above leads to interesting
\emph{general surgery constructions}, in particular connected sums,
 showing how rich the class of polar manifolds indeed is.  As an interesting example we will see that
  $\Sph^n \times\Sph^n \# \Sph^n
\times\Sph^n$ supports a polar $\SO(n)\times\SO(n)$ action of
cohomogeneity two, and metric with section $\Sph^1\times\Sph^1 \#
 \Sph^1\times\Sph^1$ of constant negative curvature. If $n \ge 3$ one can use the proof of Theorem C  to give  a geometric
 proof of the well known fact \cite{FHT} that  $\Sph^n \times\Sph^n \# \Sph^n \times\Sph^n$  is rationally
 hyperbolic.

 \smallskip

 Another rich class of polar actions is the case of general non-exceptional actions
 by a torus with 2-dimensional quotient. This gives e.g. rise to  polar actions on
connected sums of arbitrarily many copies of:  $\Sph^2\times\Sph^2$
or $\pm\CP^2$ in dimension 4,  \, $\Sph^3\times\Sph^2$ in dimension
5, and  $\Sph^4\times\Sph^2$ or $\Sph^3\times\Sph^3$ in dimension 6.
In fact in most cases each such manifold admits infinitely
 many inequivalent polar actions.

   \bigskip

  The paper is organized as follows. In Section 1 we recall
   known properties of polar actions. In Section 2
   we develop in detail the geometry of the polar group, its reflection subgroups and their chambers.
    In Section 3 we prove Theorem A and the reconstruction for general polar manifolds. Finally in
   Section 4 we discuss various applications and examples and  prove Theorem C.

\smallskip
%%%%%%%%%%%%%%%%%%%%%%%%%%%%%%%%%%%%%%%%%
%%%%%%%%%%%%%%%%%%%%%%%%%%%%%%%%%%%%%%%%%
%%%%%%%%%      Preliminaries     %%%%%%%%%%%%%
%%%%%%%%%%%%%%%%%%%%%%%%%%%%%%%%%%%%%%%%%
%%%%%%%%%%%%%%%%%%%%%%%%%%%%%%%%%%%%%%%%%

\section{Preliminaries}

As a natural generalization of \emph{polar representations},
classified by Dadok in \cite{Da}, Palais and
 Terng \cite{PT1} coined a proper isometric action $\G \times M \to M$ to be
  \emph{polar} if there is an embedded submanifold of  $M$ intersecting all $\G$ orbits
  orthogonally. It is common nowadays to work with the following extension
  since it, e.g., includes all cohomogeneity one manifolds.

\begin{definition}\label{proper}
A proper smooth $\G$ action on a (connected) manifold $M$ is called
\emph{polar} if there is a $\G$ invariant complete Riemannian metric
on $M$ and an immersion $\sigma: \Sigma \to M$ of a
 connected  complete manifold $\Sigma$, who's image intersects all $\G$ orbits
 orthogonally.
\end{definition}

To be more specific, we require that each $\G$ orbit intersects
$\sigma(M)$ and $\sigma_*(T_p\Sigma)$ is orthogonal to
$T_{\sigma(p)}(\G\cdot \sigma(p))$ for all $p\in \Sigma$. Here
$\sigma$ is referred to as a \emph{section} of the action, and $M$
simply as a \emph{polar manifold}. We will implicity assume that
$\sigma$ has no subcover section. It is an interesting question if
the immersion $\sigma$ has to be injective, at least if $M$ is
simply connected.

Note that if $\G$ acts polar, then the identity component $\G_\circ$
acts polar as well with the same section, since one easily sees that
$\cup_{g\in \G_\circ}\Sigma$ is open and closed in $M$. But
disconnected groups occur naturally, e.g., as slice representations.

\smallskip

We note that the following theorem, due to H. Boualem in the more
general situation of singular Riemannian foliations \cite{Bo} (see
also \cite{HLO} for the case of group actions), gives an alternate
definition, and motivation, to a polar manifold:
\begin{thm}\label{integrable}
The horizontal distribution associated with an isometric
 action is integrable on the regular part if and only if the action is polar.
\end{thm}

There are   trivial polar actions where $\G$ discrete and hence
$\Sigma=M$, as well as the case where $G$ acts transitively and
hence $\Sigma$ a point. We call such actions {\it improper} polar
actions. We will usually assume that the polar action is proper,
although improper ones sometimes occur, e.g., as slice representations
of the $\G$ action.

\smallskip
In the following we denote by $M_\circ$ the regular points in $M$,
and by $\Sigma_\circ$ those points whose image is in $M_\circ$.
Recall that $M_\circ$ is open and dense in $M$ and
$M_\circ/\G$ is connected. For any $q\in M$, the isotropy $\G_q$
acts via the so called slice representation on the orthogonal
complement of the orbit $T_q^\perp:=(T_q(\G\cdot q))^\perp$. The
image $S_q:=\exp_q(B_\e )$ in $M$ of a small ball $B_\e\subset
T^\perp_q$
 is called the slice at $q$. By the slice theorem
$\G\times_{\G_q}B_\e$ is equivariantly diffeomorphic to a tubular
neighborhood of the orbit $\G\cdot q$.

\smallskip

 Fix a section $\sigma: \Sigma \to M$, a point $p\in\Sigma_\circ$, and let $\H =
\G_{\sigma(p)}$ be the isotropy group  corresponding to the
principal  orbit $\G\cdot \sigma(p)$. Let $\G_{\sigma(\Sigma)}
\subset \G$ be the stabilizer of the image, and $Z_{\sigma(\Sigma)}$
the centralizer (i.e. fixing $\sigma(\Sigma)$ pointwise). We refer
 to $\Pi: = \G_{\sigma(\Sigma)}/Z_{\sigma(\Sigma)} $ as the \emph{polar group} associated to the section $\sigma$
 (the corresponding group was called the ``generalized Weyl" group in
 \cite{PT1}). Note that for improper polar actions $\Pi = \G$ when $\Sigma = M$
and $\Pi = \{1\}$ when $\Sigma$ is a point.

\smallskip

The usual properties of polar actions for which the section is
embedded, easily carry over to the immersed case, though care needs
to be taken since $\sigma$ may not be injective. For the convenience
of the reader, we indicate the details.

\begin{prop}\label{genproperties}
Let $\G$ act on $M$ polar with section $\sigma\colon \Sigma\to M$
and polar group $\Pi$ and projection $\pi\colon M\to M/\G$. Then the
following hold:
\begin{enumerate}
\item[(a)] The image $\sigma(\Sigma)$ is totally geodesic and  determined by the
slice of any principal orbit.
  \item[(b)] The polar group $\Pi$ is a discrete subgroup of $\N(\H)/\H$ and acts isometrically and properly discontinuously on
   $\Sigma$ such that $\sigma$ is equivariant. Furthermore,
  $\pi\circ\sigma$ induces an isometry between $\Sigma/\Pi$ and
  $M/\G$ and thus $M/\G$ is an orbifold.
   \item[(c)] $\Sigma_{\circ}:=\sigma^{-1}(M_\circ)$ is open and dense in $\Sigma$ and $\sigma$ is injective on
    $\Sigma_{\circ}$. The polar group $\Pi$ acts freely on
$\Sigma_{\circ}$ and $\pi\circ\sigma$ induces an isometric
covering from each component of  $\Sigma_{\circ}$ onto $
M_{\circ}/ \G$.
 \item[(d)] The slice representation of an isotropy group $\G_q ,\ q\in\sigma(\Sigma),$ is a
polar representation, and for each $p\in\sigma^{-1}(q)$ the
linear subspace $\sigma_{*} (T_p \Sigma)\subset T_q^\perp
\subset T_qM$ is a section with  polar group the (finite)
isotropy group $\Pi_p$.
\end{enumerate}
\end{prop}
\begin{proof}
 If $p\in\Sigma_\circ$, then for dimension reasons
$T^\perp_{\sigma(p)}=\sigma_*(T_p\Sigma)\subset T_{\sigma(p)}M$.
Thus, if  $c$ is a geodesic in $\Sigma$ with $\gamma(0)=p$,
$\sigma(c)$ is horizontal on the regular part and hence a geodesic
in $M_\circ$ since $\pi\colon M_{\circ}\to M_{\circ}/\G$ is a
Riemannian submersion. The non-regular points on $\sigma(\gamma)$
are isolated. Indeed,  the set of regular points is clearly
non-empty and open. If $q=\gamma(t_0)$ lies in its closure, the
slice theorem at $q$, since $\sigma(\gamma)\subset S_q$, implies
that the isotropy groups at $\sigma(\gamma)(t_0-\e)$ and
$\sigma(\gamma)(t_0+\e)$ are the same, which proves our claim. This
implies that $\Sigma_\circ$ is dense in $\Sigma$ and hence
$\sigma(\Sigma)$ is totally geodesic.
 Since $\sigma$ has no
subcovers, $\sigma_*(T_p\Sigma)$ determines the immersion. Notice
this also implies that if $g\in\G$ satisfies
$g_*(\sigma_*(T_p\Sigma))=\sigma_*(T_q\Sigma)$ for some
$p,q\in\Sigma$, then $g\in\G_{\sigma(\Sigma)} $ with $gp=q$.
Similarly, $\sigma\colon\Sigma_\circ\to M_\circ$ is injective and
$\Pi$ acts freely on $\Sigma_\circ$.

 By the
slice theorem the principal isotropy at $\G_{\sigma(p)}=\H$ acts
trivially on $S_{\sigma(p)}\subset\sigma(\Sigma)$ and since $\Sigma$
is totally geodesic,
 $\H$ acts trivially on $\sigma(\Sigma)$.
In fact, by the slice theorem  again, the centralizer
$Z_{\sigma(\Sigma)}=Z_{\sigma(\Sigma_{\circ})}=\H$. Since $\H$ is
thus the ineffective kernel of the action of $\G_{\sigma(\Sigma)} $
on $\sigma(\Sigma)$, $\H$ is normal in $\G_{\sigma(\Sigma)} $, i.e.,
$\G_{\sigma(\Sigma)} \subset\N(\H)$. Next, let $h\in
\G_{\sigma(\Sigma)}$ be close to the $\Id$. By the slice theorem the
orbits  in a tubular neighborhood of the regular point $\sigma(p)$
intersect the slice in a unique point and thus
$h\sigma(p)=\sigma(p)$, i.e. $h\in \H$. Thus $\H$ is open in
$\G_{\sigma(\Sigma)} $ and hence $\Pi$ is discrete in $\N(\H)/\H$.

 The polar group $\Pi$ acts on $\sigma(\Sigma)$ by definition and
preserves $\sigma(\Sigma_{\circ})\subset M_{\circ}$. Since
$\sigma\colon \Sigma_{\circ}\to M_{\circ}$ is injective,  the action
lifts to an isometric action on $\Sigma_{\circ}$ such that $\sigma$
is $\Pi$ equivariant. Since the action is isometric and
$\Sigma_{\circ}$ is dense, $\gamma\in\Pi$ extends to a metric
isometry on $\Sigma$ and hence a smooth isometry. By the same
argument as above, $\Pi\subset\Iso(\Sigma)$ is a discrete subgroup.
That it acts properly discontinuously is a direct consequence of
properness of the $\G$ action.  Indeed, since the $\G$ orbits are
closed, so are the $\Pi$ orbits, hence preventing accumulation
points.

We now show that   $\pi\circ\sigma$ induces an isometry
$\Sigma_\circ/\Pi\simeq M_\circ/G$, which by continuity implies
$\Sigma/\Pi\simeq M/G$. Let $p_i\in\Sigma_\circ$ with
$\sigma(p_2)=g\sigma(p_1)$. Then $g_*\sigma_*(T_{p_1}\Sigma)=
\sigma_*(T_{p_2}\Sigma)$. Hence $g\in\G_{\sigma(\Sigma)} $ and
 thus there exists a unique $\gamma\in\Pi$ with $p_2=\gamma p_1$.
 Furthermore, $\pi\circ\sigma \colon \Sigma_\circ\to M_\circ/\G$ is a local
 isometry, in fact a covering onto $M_\circ/\G$.

 To see that the slice representation of $\K:=\G_{\sigma(p)}$ on $T^\perp_{\sigma(p)}$ is polar
 with sections $\sigma_{*} (T_p \Sigma) \subset T_qM$,
 take a geodesic $c$ in $\Sigma$
 starting at $p$. Then by the slice theorem the isotropy along the geodesic $\sigma(c(0,\e))$ is
 constant, and is equal to the isotropy of the slice representation along
 $t\sigma_*(c'(0))$ for $t\in(0,\e)$. As $t\to 0$, it follows that the $\K$ orbit $\K\cdot (t\sigma_*(c'(0)))$  is orthogonal to  $\sigma_{*} (T_p
 \Sigma)$. In addition, we need to show that each $\K$ orbit  meets  the linear subspace $\sigma_{*} (T_p
 \Sigma)$. Equivalently, $\K\cdot s=\G\cdot s\cap S_{\sigma(p)}$ with $s\in S_{\sigma(p)} $ meets $\sigma(B_\e(p))\cap S_{\sigma(p)}$.
 To see this, pick a regular point $r\in\Sigma_\circ$ near
 $p$. Now connect  the two orbits  $\G\cdot \sigma(r)$ and $\G\cdot
 s$ by a minimal geodesic $c$ in $M$, which we can assume starts at
 $\sigma(r)$ and is hence tangent to $\sigma_{*} (T_p
 \Sigma)=T_{\sigma(p)}(\G\cdot {\sigma(p)})^\perp $. Thus $c\subset\sigma(\Sigma)$,  which implies that
   $\G\cdot
 s$ and hence $\K\cdot s$
 meets $\sigma(\Sigma)$. Finally, the polar group of  $\sigma_{*} (T_p
 \Sigma)$ is clearly $\Pi_p$.
 \end{proof}

We will call $\Pi_p$ the {\it local polar group}.  Notice that
$\Pi_p$  is in general  smaller than the stabilizer
$\Pi_{\sigma(p)}$ (if $\sigma^{-1}(q)$ has more than one point), and
that $\Pi$ acts transitively on $\sigma^{-1}(q)$.

\smallskip

 It is interesting to note that for a general group
action, the quotient $M/\G$ is an orbifold if and only if the action
is infinitesimally polar, i.e., all  slice representations are polar
\cite{LT}.

\smallskip

If a polar action by $\G$ acts linearly and isometrically on a
vector space $V$, it is called a polar representation. If $\G/\K$ is
a symmetric space with $\K$ connected, the isotropy action of $\K$
on the tangent space is well known to be polar, the section being a
maximal abelian subspace  $ \fa\subset\fk^\perp\subset\fg$. Such
representations are called {\it s-representations}. Notice that the
cohomogeneity of the action of $\K$ is the rank of the symmetric
space. The action of $\K$ of course also induces a polar action on
the unit sphere in $V$, with cohomogeneity one less than the rank.

\smallskip

 The remarkable result by Dadok \cite{Da}
(see also \cite{EH1} for a conceptual proof if the cohomogeneity is
at least 3) is a converse. For this, first recall that two actions
with the same orbits are called orbit equivalent.

\begin{prop}\label{linear} Let $\G$ be a connected Lie group acting linearly
on $V$.
\begin{enumerate}
\item[(a)] The $\G$ action on $V$ is polar if and only if it is orbit equivalent to an
s-representation.
 \item[(b)] If the $\G$ action is polar, the polar group is a finite group generated by  linear reflections.
\end{enumerate}
\end{prop}

 Notice that orbit equivalent actions have the same Weyl group (as an action on $\Sigma$). Hence part (b) follows from (a)
since this fact is well known for symmetric spaces. In this case the
polar group coincides with the usual Weyl group of a symmetric
space. We reserve the word Weyl group for this particular kind of
polar group. Thus, if $\K=\G_{\sigma(p)}$,  the slice representation
of $\Ko$ is an s-representation and its polar group $\Pi_p\cap \Ko$
satisfies all the well known properties of a Weyl group, and in
particular is generated by linear reflections with fundamental
domain a (Weyl) chamber. The full local polar group $\Pi_p$ is a
finite extension of this Weyl group.

\begin{rem*}
If two connected groups $\K$ and $\K'$ are orbit equivalent polar
representations, with $\K$ being an s-representation, then it
 was shown in \cite{EH1} that $\K'\subset\K$ and in \cite{EH2,Ber,FGT} one
 finds a list of all connected subgroups of s-representations which are orbit equivalent.
\end{rem*}

\smallskip

For metrics we have the following natural but highly nontrivial
result due to R. Mendes \cite{Me}:

\begin{thm}[Metrics]\label{metrics}
Let $M$ be a polar $\G$ manifold with section $\sigma\colon\Sigma\to
M$ and polar group $\Pi$. Then for any $\Pi$ invariant metric
$g_\circ$ on $\Sigma$ there exists a $\G$ invariant metric $g$ on
$M$ with section $\Sigma$, such that $\sigma^*(g)=g_\circ$.
\end{thm}

\bigskip

%%%%%%%%%%%%%%%%%%%%%%%%%%%%%%%%%%%%%%%%%
%%%%%%%%%%%%%%%%%%%%%%%%%%%%%%%%%%%%%%%%%
%%%%%%%%%      Geometry of the quotient $M/\G$     %%%%%%%%%%%%%
%%%%%%%%%%%%%%%%%%%%%%%%%%%%%%%%%%%%%%%%%
%%%%%%%%%%%%%%%%%%%%%%%%%%%%%%%%%%%%%%%%%

\bigskip

\section{Polar group, Reflection subgroups and Chambers}

\bigskip

To begin a more detailed analysis of the action by the polar group
we derive more information about
 the orbit space $M/\G = \Sigma/\Pi$.
 This in particular will reveal the existence of an important normal subgroup $\Rho \subset \Pi$
 generated by \emph{reflections},
 which typically is non-trivial. In particular, the presence of \emph{singular orbits}, i.e.,
 orbits with dimension lower than the regular orbits, will yield the existence of a non-trivial
 normal reflection subgroup $\Rho_s \subset  \Pi$. The associated concept of
 \emph{mirrors} lead to the presence of open \emph{chambers}, whose length space completion and
  the isotropy groups along them  are the crucial ingredients for the recognition and reconstruction.

\bigskip

 We refer to polar actions with no singular points simply as {\it exceptional polar} actions. For those we have the following
characterization:

\begin{thm}[Exceptional Polar]\label{exceptional}
An action $\G \times M \to M$ without singular orbits is polar if
and only if $M$ is equivariantly
 diffeomorphic to an associated bundle $\G/\H \times_{\Pi} \ \Sigma$, where $\Pi$ is a discrete subgroup
  of $\N(\H)/\H$ acting isometrically and properly discontinuously on a Riemannian manifold
  $\Sigma$. The image of $(e\H)\times\Sigma\subset \G/\H\times\Sigma$ is a section and  $\Pi$ is  the polar group.
\end{thm}
\begin{proof}
We start with the Riemannian manifold $M=\G/\H \times_{\Pi} \
\Sigma$  on which $\G$ acts on the first component from the left.
Here $(g\H,p)$ is equivalent to $(gn^{-1}\H,np)$ for $n\in\Pi\cdot
\H\subset\G$. Since $\N(\H)/\H$ acts freely on $\G/\H$ on the right,
$\G/\H \times_{\Pi} \ \Sigma$ is a manifold. The image of
$({e},\Sigma)$ is a section if and only if the metric is induced by
a warped product metric of a family of left $\G$ right $\H$
invariant metrics on $\G/\H$ with the given metric on $\Sigma$. As
for the isotropy groups of this action we have
$\G_{[(eH,p)]}=\Pi_p$, where $\Pi_p$ is the stabilizer of $p$. Since
$\Pi_p$, being a discrete subgroup of $\O(T_p\Sigma)$, is finite, all
orbits are exceptional.

Conversely, assume that $\G$ acts polar on $M$ with section
$\Sigma$, polar group $\Pi$ and with all orbits exceptional.  The
section through every point is unique and hence $\sigma$ is
injective. We claim that $\pi\colon \G\times\Sigma\to M,\, (g,p)\to
g\sigma(p)$ induces an equivariant diffeomorphism of $\G/\H
\times_{\Pi} \ \Sigma$ with $M$. Clearly $\pi$ is onto. If
$p,p'\in\Sigma$ and $ g,g'\in\G$ with $g\sigma(p)=g'\sigma(p')$,
then $(g')^{-1}g\sigma(p)=\sigma(p')$.  By uniqueness of the
sections, $(g')^{-1}g=\gamma\in\Pi$ and since $\sigma$ is injective
$p'=\gamma p $ and $g'= g\gamma^{-1}$. This proves our claim.
\end{proof}

\smallskip

\begin{rem*}
(a)  Thus any discrete subgroup of a Lie group is  the polar group
of some (exceptional) polar manifold.

(b) It follows that a principal $\G$ bundle $P \to B$ with $\G$
compact is polar if and only if it has finite structure group. Note
that in this case a  section is not a section in the principal
 bundle sense, but rather a ``multiple valued" section: It is locally a section, and in fact a cover
  of the base.

  (c) Observe that  an exceptional $\G$ manifold $M$ (with $\G$ compact) has an invariant
metric of non-negative curvature if and only
 if $\Sigma$ has a $\Pi$ invariant metric of nonnegative curvature \end{rem*}

\bigskip

 In the remainder of the
paper we will be interested in polar manifolds which are not
exceptional.

\smallskip

We will use the following notation: $M_{\circ}$ is the
\emph{regular} part of $M$, i.e., the union of all principal orbits,
$M_e$ is the \emph{exceptional} part of $M$, i.e.,
 the union of all non-principal orbits with the same dimension as the
 principal orbits, and $M_s$ is the \emph{singular} part of $M$, i.e.,
  the union of all lower dimensional orbits.     $M_{\circ}$ is open and dense in $M$ and $M/\G$ . Recall that
   $\Sigma_{\circ} = \sigma^{-1} (M_{\circ})$ is also open and dense in $\Sigma$.

The maximal leaf of the (smooth) horizontal distribution in $
M_{\circ}$  through a
 point $p\in M_{\circ}$ is the connected component
  of $\Sigma_{\circ}$  corresponding to the (unique)
  section $\Sigma$  with $p\in \sigma(\Sigma)$. From now on, we will denote this component by $\Sigma_{\circ}(p)$.
   Moreover,  $\sigma: \Sigma_{\circ}(p) \to M$ is 1-1  and
$\pi\circ\sigma\colon \Sigma_{\circ}(p) \to M_\circ/\G$ is an
isometric covering.

  We  denote by $M^{\K}$ the fixed point
set of $\K$ (a totally geodesic submanifold) and by $M_{(\K)}$ the
orbit type, i.e. the union of all orbits with isotropy conjugate to
$\K$. The image of any set in $M$ we denote by a $*$, e.g.,
$M^*=M/\G$ and $M_{(\K)}^* =M_{(\K)}/\G$.

\smallskip

\begin{rem*}
In general, recall that the \emph{core} of a $\G$ manifold $M$ is
the subset $C(M) \subset M^{\H}$ of the fixed point set   of  $\H$
having non-empty intersection with the regular part $M_{\circ}$ (cf.
\cite{GS2}), and the \emph{core group} $\N(\H)/\H$ its stabilizer.
 Any connected component $C_{\circ}(M)$ of the core together with the induced action by
  $\G^*:=\G_{C_{\circ}(M)} / \H \subset \N(\H)/\H$  has the same significance as the
  core itself. Since clearly $\sigma(\Sigma) \subset C_{\circ}(M)$ for a core component
   $C_{\circ}(M)$, it follows that $M$ is polar if and only if the action of $\G^*$ on $C_{\circ}(M)$ is polar, and in
   this case the polar groups of $C_{\circ}(M)$ and of $M$ coincide.
   The principle isotropy of $\G^*$ is trivial, but the group $\G^*$ itself is in general not
   connected. Notice also that singular isotropy of $\G$ can become
   exceptional for $\G^*$.
\end{rem*}

\smallskip

We first recall some facts about the geometry of $M/\G$ that follow
from the slice theorem. If $S$ is a slice at $q\in M$ and $\K=\G_q$,
then $M_{(\K)}\cap S = M^{\K}\cap S$. A neighborhood of the point
$p^*=\G/\K$ in $M^*$ is isometric to $S/\K$ and hence $ M^{\K}\cap
S$ is isometric to a neighborhood of $p^*$ in $M_{(\K)}^*$.  In
particular, each component of $M_{(\K)}^*$ is a totally geodesic
(but typically not complete) submanifold in the quotient $M^*$.
Furthermore, if $c_{|[a,b]}$ is  a minimal geodesic between orbits,
then $\G_{c(t)}=Z_{\Im(c)}$ for all $t\in (a,b)$  since otherwise
$c$ could be shortened. Hence $M_\circ^*$ is convex, and $M^*$ is
stratified by (locally) totally geodesic orbit types.

 Let $\K=\G_q$ be an isotropy group where only the principal isotropy group $\H$ is
  smaller. We call such an isotropy group an \emph{almost minimal} isotropy
group. Then the action of $\K$ on the sphere orthogonal to
$T_q(M^{\K}\cap S)\subset T_qS $ has only one orbit type, namely
$\H$. As is well known, see, e.g., \cite{Br},p.198,  this implies that
the action of $\K$ on the sphere is either transitive, or it acts
freely, and $\K_\circ$ is one of the Hopf actions.
 But if $\G$ is a
polar action, the slice representation by $\K$ as well as $\K_\circ$
is polar. Since the horizontal space for the Hopf actions is not
integrable, it follows that if $\K$ acts freely, then it must be
finite, i.e., $\G\cdot q$ is an exceptional orbit.

 If $\K$ acts
transitively, $q^*\in M^*$ is a boundary point and $M^{\K}\cap
S\subset M^*$ is part of the boundary in $M^* = \Sigma^*$. Thus
$M^*_{(\K)}$ is part of the boundary and a connected component of
$M^*_{(\K)}$ is called an {\it open face} and its closure a  {\it
face} of the boundary.

\smallskip

We now observe that $M^*_s\subset
\partial M^*$. Indeed, if $\L$ is a singular isotropy of $\G$, the slice representation of $\L_\circ$
is polar and hence has singular isotropy. By induction, $\L$
contains an almost minimal $\K$ with $\K/\H$ not finite and hence
$M^*_{(\L)}$ lies in the closure of $M^*_{(\K)}$. Exceptional orbits
can also be part of the boundary: If $\K$ is almost minimal with
$\K/\H\simeq \Z_2$ and if the fixed point set of the involution has
codimension $1$ in the slice, then $M^*_{(\K)}$ is a face. An
exceptional orbit for which $\K$ contains such an involution lies in
$\partial M^*$ as well.  Of course there can be interior exceptional
orbits if $M$ is not simply connected, but their codimension is at
least $2$. We remark that $\partial M^*$ is also the boundary of
$M^*$ in the sense of Alexandrov geometry. Summarizing:
\smallskip

\begin{itemize}
\item $M^*$ is convex and stratified by totally geodesic orbit
types.
\item $M_s^*\subset \partial M^*$ and interior exceptional orbits
have codimension at least 2.
\end{itemize}

\bigskip
As for the geometry of the discrete group action  of $\Pi$ on
$\Sigma$, (see, e.g., \cite{Dav}),  it is common to consider the
normal subgroup $\Rho$ consisting of {\it reflections}. Here $r$ is
called a {\it reflection} if the fixed point set $M^r$ has a
component  $\Lambda$ of codimension 1, a totally geodesic reflecting
hypersurface called a \emph{mirror} for $r$.  Notice that we do not assume that $\Lambda$
separates $\Sigma$ into two  parts, as in \cite{Dav}. In particular,
$M^r$ can contain several reflecting hypersurfaces (and components
of lower dimension).
 Note though that if $\Sigma$ is simply connected,
degree theory implies that reflections separate. On the other hand,
sections are typically not simply connected, even if $M$ is.
 See
also \cite{AA2} for a general discussion of reflection groups,
 and reflections groups in which reflections separate.

\smallskip

We denote the set of all mirrors in $\Sigma$ by $\mathfrak{M}$. A
connected component of the complement of the
 set  of all mirrors in $\Sigma$ is called an \emph{open chamber}, and denote by $c$. The
closure $C = \bar c$ (in $\Sigma$) of an open
 chamber will be referred to as a {\it closed chamber}, or simply as a chamber of
 $\Sigma$.
 If $ \Lambda \in \mathfrak{M}$, we call the components of the intersection  $ C \cap \Lambda$ which
  contain open subsets of $\Lambda$
  a \emph{chamber face} of $C$. We can provide each chamber face
  with a label $i\in I$ and  will denote the face by $F_i$ and the
  corresponding reflection by $r_i$. Note though that different
  faces can correspond to the same reflection, i.e., possibly
  $r_i=r_j$.

 $\Rho$ (and hence $\Pi$) acts  transitively on the set of open
  chambers. In our generality, however, there may be non trivial elements $\gamma\in\Pi$ (even in $\Rho$)
  with $\gamma(C)=C$. We denote the subgroup of such elements by
  $\Pi_C$ (respectively $\Rho_C$). Clearly then $\Sigma/\Pi=C/\Pi_C$ and one easily sees that the
  inclusion $\Pi_C\subset\Pi$ induces an isomorphism  $\Pi_C/\Rho_C\simeq
  \Pi/\Rho$.

Recall that we denote by $\Sigma_{\circ}(p)$ the component of
$\Sigma_\circ$ containing $p$. We thus have:

\smallskip

\begin{itemize}
\item $C=$cl$(\Sigma_{\circ}(p))$ is locally the union of  convex sets with
 $\partial C=\cup_{i\in I}F_i$. Furthermore, $M^*=C/\Pi_C$, $\partial
 M^*=\partial C/\Pi_C$, and $\Pi_C/\Rho_C\simeq \Pi/\Rho$.
\item $\Pi_C$ preserve strata types, and in particular
 permutes chamber faces, possibly preserving some as well.
\end{itemize}

\smallskip

   A component of an  isotropy type $\Sigma_{(\Gamma)}$ of $\Pi$ (or of $\Rho$) is contained in
   $\Sigma^{\Gamma'}$ (as an open convex subset) with $\Gamma'$
  conjugate to $\Gamma$ since a slice at $p\in \Sigma_{(\Gamma)}$ is a neighborhood of $p$.
   It is contained
   in a mirror
   if $\Gamma$, and hence $\Gamma'$,
  contains a reflection $r\in\Gamma'$ since then $M^{\Gamma'}\subset M^r$ . Of course different components
  of $\Sigma_{(\Gamma)}$ can be contained in different mirrors, or even in the same mirror.
  In general, there can be isotropy in the interior of $c$, but the codimension of
  $M^\Gamma$ is then at least 2.
 Let
  $p\in M^\Gamma$ be contained in the $\Gamma$ isotropy type of $\Pi$ and
  $q=\sigma(p)$. Then $\K:=\G_q$ acts polar  on the slice with
  section $\sigma_*(T_p\Sigma)\subset T_q^\perp$ and polar group $\Pi_p\simeq\Gamma$. The fixed point
  set of $\K$ in the slice is $T_q(M^{\K}\cap S)\subset
  T_q^\perp$, and since every $\K$ orbit meets
  $\sigma_*(T_p\Sigma)$,  it follows that   $T_q(M^{\K}\cap S)\subset\sigma_*(T_p\Sigma)$,
  in fact for every $p\in\sigma^{-1}(q)$. Furthermore,
  $\K$ acts polar on the unit sphere $\Sph^\ell$ in $T_q(M^{\K}\cap
  S)^\perp$, without any fixed vectors. One easily sees that this
  implies that $\Gamma$ also has no fixed vector in $\Sph^\ell\cap\sigma_*(T_p\Sigma)$ and
  hence $T_q(M^{\K}\cap S)$ is also the fixed point set of $\Gamma$, in other words
   $\sigma(U)=M^\K\cap S$ for
  some neighborhood $U$ of $p$ in $M^\Gamma$.
Thus
\smallskip
\begin{itemize}
  \item $\pi\circ \sigma\colon\Sigma\to M^*$ takes orbit types of $\Pi$ to orbit
  types of $\G$ and chamber faces of $C$ to faces of $\partial
  M^*$. Furthermore, $\Pi_C$ acts transitively on the chamber
  faces in the inverse images of a face, and possibly
  non-trivial on a chamber face as well.
  \item For all points $p\in\Sigma$ lying in a component of an orbit
  type, the isotropy $\G_{\sigma(p)}$ is constant, and the slice
  representations are canonically isomorphic as well.
  \end{itemize}
  \smallskip
  The constancy of the isotropy will be crucial in our
  reconstruction.

  \smallskip

  Let $F$ be a face, $\Pi_{C,F}\subset\Pi_C$ the subgroup that stabilizes
  it, and $\Pi_{C,p}$ the isotropy at a point $p\in F$. Since $C/\Pi_C=M^*$, the orbit types of the
  action of $\Pi_{C,F}$ on $F$ are taken to the orbit types in
  $M^*$. Thus there exists an open and dense set of points in $F$, the regular points of the
  action of $\Pi_{C,F}$ on $F$, along which  the $\G$-isotropy is
  constant equal to $\K$. We call this the {\it generic} isotropy of
  $F$. Again from
  $C/\Pi_C=M^*$ it follows that:
  \smallskip
  \begin{itemize}
  \item If $ F$ is a face with generic isotropy $\K$, then at points $p$ in the interior of $F$ the isotropy
  $G_{\sigma(p)}=\K\cdot\Pi_{C,p}$.
  \end{itemize}
Such a generic isotropy is also an almost minimal isotropy, in
particular $\K/\H=\Sph^\ell$, and at a generic point $p$ the $\Pi$
isotropy $\Pi_p\simeq\Z_2$, generated by the reflection $r$ in the
face $F$. Of course $r\in\K$, in fact $r$ is the unique involution
in $(\N(\H)\cap\K)/\H$.

\smallskip

 It will be important for us to also consider the larger completion $C'$ of $c$ obtained from
  the intrinsic length metric on $c$, and the induced surjective map $C' \to
  C$, which induces a one-to-one correspondence between orbit types
  (but not necessarily components of orbit types). The advantage of considering $C'$ as opposed to $C$
  is that for $p\in\partial C$ the
  tangent cone $T_pC$ minus $p$ may not be connected, but these
  components are ``separated" in $C'$, i.e., $T_pC'$ minus $p$ is now
  indeed
  connected. In fact

  \begin{itemize}
  \item Each point of the boundary $\partial C'$
  has as a convex neighborhood the cone over the fundamental
  domain of the  reflection subgroup of the local polar group
  acting  on the normal sphere to the orbit.
  \end{itemize}

Due to this phenomena, $C$ may not be convex (see the example
below),
  though it is clearly
  a union of locally convex sets.
  This also means that although $C/\Pi_C$ is an Alexandrov space,
  $C$ itself may not be.

$\Pi_C$ acts on the interior of $C$ and  induces an isometric action
on the completion $C'$ with $C'/\Pi_C = C/\Pi_C$.
 We refer to it as the lift to $C'$.
The inverse image of chamber faces in $C$ are called chamber faces
of
  $C'$.
We now construct a Riemannian manifold $\Sigma'$ on which a Coxeter
group acts and in which $C'$ will be a convex subset. To do this,
first observe that the angle between any two chamber faces $F_i$ and
$F_j$
  along a component of a codimension 2 intersection (if it exists)
determine the order $m_{ij}$ (infinity if $F_i \cap F_j =
\emptyset$) of the rotation $r_i r_j$. Although the intersection may
have more than one component,  we claim that these integers are the
same for each component. Indeed, if $p$ is a generic point in a
component of $F_i\cap F_j$, then the action of $\G_{\sigma(p)}$ on
the sphere $\Sph^\ell$ normal to the orbit type is cohomogeneity
one.  The angle is the length of the interval
$\Sph^\ell/\G_{\sigma(p)}$, which is equal to the length  of the
section (a great circle) divided by the Weyl group. Furthermore, the
two non-regular isotropy groups are the generic isotropy groups of
$F_i$ and $F_j$. By constancy, the generic isotropy group of the
faces are the same on any other component of $F_i\cap F_j$. But this
implies that the Weyl group is the same as well since it is
generated by the two reflections in the non-regular isotropy.  We
can thus associate to $C$ a Coxeter group: $\M:=\{r_i\mid i\in I
\text{ with } r_i^2=1,\, (r_i r_j)^{ m_{ij}}=1\} $.  Notice that
there is a natural homomorphism $\M\to\Rho$ which is onto since
$\Rho$ is generated by the reflections $r_i$, Thus $\Rho$ is the
quotient of a Coxeter group. For polar actions, the corresponding
Coxeter matrix has special properties:
 As is
well known, the quotient $\Sph^\ell/\G_{\sigma(p)}^0$  is an
interval of length $\pi/n$ with $n=2,3,4,6$. The same is true for
$\Sph^\ell/\G_{\sigma(p)}$  if there are no exceptional orbits (see,
e.g., again \lref{linearisotropy} below), and possibly half as long
in general.

\smallskip

As usual, we can use the Coxeter group $\M$ to define a smooth
Riemannian manifold $\Sigma'$ on which $\M$ acts by taking the
quotient of $\M\times C'$ by the equivalence relationship
$(gr_i,x)\sim (g,x)$ for $x\in F_i,\ g\in\M$ which identifies
``adjacent" faces. One easily sees that $\Sigma'$ is smooth since
$\Sigma$ is, and $\M$ acts isometrically on $\Sigma'$, with
fundamental chamber $C'$, and simply transitively on chambers (in
particular $\M_{C'}$ is trivial). We thus have a natural map
$\Sigma'\to\Sigma$, taking $\M$ translates of $C'$ to $\Rho$
translates of $C$, which is an isometric covering, equivariant with
respect to the quotient homomorphism $M\to \Rho$. We can think of
$\Sigma'$ as the ``universal" section associated to $C$.

  Note that although combinatorially two polar manifolds may be the same, i.e., same $\M$ structure,
  as manifolds they of course depend on the topological structure of the fundamental domain $C$. Using this for
  example will result in very different polar manifolds via our reconstruction.

\smallskip

 These concepts are illustrated in the following

 \begin{example}

Consider the group $\Rho = \D_3 = \langle\r_{\circ}, \r_3\rangle$
acting on $\Sph^2$ as well as on $\RP^2$, where $\r_{\circ}, \r_3$ are
reflections in two great circles making an angle $\pi/3$:

\begin{itemize}

\item
On $\Sph^2$ there are three mirrors and six open chambers. Their
 closure is the orbit space $\Sph^2/\Rho$, a spherical biangle
 with angle $\pi/3$. There are two faces each of which are also
 a chamber face. Their intersection is the intersection of
 mirrors and coincides with the fixed point set $\Fix(\Rho)$.

\item On $\RP^2$ there are three mirrors and three open chambers.
 The closure of an open chamber is a spherical biangle with
 angle $\pi/3$ where the two vertices have been identified! The
 stabilizer $\Rho_C$ of a chamber has order two, fixes the ``mid
 point" of $C$ and rotates $C$ to itself, mapping one chamber
 face to the other. The orbit space has one face with one
 singular
  point, the fixed point of $\Rho$ and one interior singular
  point, the fixed point of $\Rho_C$. ($\tilde C$ is a spherical
  biangle with angle $\pi/3$ and $\Rho_C = \Z_2$ acts by
  rotation).

In this case, the reflection group obtained by lifting the
reflections in $\RP^2$ to reflections in $\Sph^2$ does not
contain the antipodal map, the deck group. If we extend it by
the antipodal map we get the same action on $\RP^2$ and the
 orbit spaces are of course the same as well.
\end{itemize}

Now consider the manifold $M$ with boundary consisting of six
circles, obtained from $\Sph^2$ by
 removing a small $\epsilon$ ball centered at the soul point of each chamber, and let $N$ denote
 its quotient by the antipodal map, having three boundary circles. Clearly,  $\Rho = \D_3$ preserves
  $M$ and induce an action on $N$.

\begin{itemize}

\item
 $\Rho = \D_3$ acts as a reflection group on the doubles D($M$)
  and D($N$) of $M$ and $N$ respectively. On D($M$) which is an
  orientable surface of genus five,  $\Rho_C = \Z_2$ acts freely
  on the chamber
   $C$. However, on its quotient D($N$),  $\Rho_C = \Z_2$ acts
   freely on the open chamber but fixes two points of its
   closure (on the corresponding $\tilde C$ is isometric to a
   chamber of D($M$), and $\Rho_C = \Z_2$ acts freely on it).
 \end{itemize}

 Now consider the group $\Pi = \Z_2 \times \Z_3$ acting on $\Sigma$ the double of a pair of pants,
  i.e., on a surface of genus two. Here $\Z_2$ is given by the reflection in the double, i.e.,
  mapping one pair of pants to the other, and $\Z_3$ rotates each pair of pants. In this example
  $\Rho = \Z_2$, there are two chambers a pair of pants, and the stabilizer groups of a chamber
  $C$ are $\Rho_C = \{1\}$ and  $\Pi_C = \Z_3$. In this example, $\Pi_C$ has two fixed points
  (as it does on the corresponding $\tilde C$).

 If we remove a disc at each fixed point of $\Z_3$ and join the two boundary circles by a cylinder,
 we get a new manifold, a surface of genus four with the same group acting, but where $\Pi_C$ acts
  freely on the chamber.

\begin{itemize}
\item
Note that the first example above can be realized concretely as
sections with polar groups actions: Let $\G = \SO(3)$ acting on
$\Sph^4$ with cohomogeneity one and orbit space an interval of
length $\pi/3$. Then the suspended action on $\Sph^5$ is polar
with section $\Sph^2$ and polar group $\Pi = \Rho= A_2$ the
dihedral
 group of order $6$. Its quotient action on $\RP^5$ has section $\RP^2$ with the same polar group but here $\Pi_C = \Z_2$
 as explained above.
 \end{itemize}

\end{example}

 \bigskip

One can refine the description of the reflection group $\Rho$. In
fact, $\Pi$ contains two natural normal reflection groups, $\Rho_s$
and $\Rho_e$, the \emph{singular} reflection group and the
\emph{exceptional} reflection group respectively. Here $\Rho_s$ is
generated by the reflections corresponding to faces consisting of
singular orbits, and $\Rho_e$ by the reflections corresponding
 to faces consisting of exceptional orbits. Clearly, $\Rho = \Rho_s \cdot
 \Rho_e$. The discussion above for chambers  can be carried our for either of the subgroups $\Rho_s$ and $\Rho_e$. This is particularly useful
 for $\Rho_s$, where we will use corresponding notation, such as $c_s, C_s$ and $C_s'$, and $\M_s$
  for the
 associated Coxeter group.

 \smallskip

We call a polar action {\it Coxeter polar} if there are no
exceptional orbits and $\Pi_C$ is trivial, in particular
$\Pi=\Rho=\Rho_s$ is a reflection group. Note though that $\Pi$
itself may not be an abstract Coxeter group, though it is a quotient
of  the Coxeter group $\M$ associated to $C$. Since $C/\Pi_C= M^*$,
we have:

\begin{prop} Let $\G$ be a Coxeter polar action on $M$ with polar group $\Pi=\Rho_s$. Then
 $\Pi$
acts freely and simply transitively on the set of chambers. The
chamber $C$ is a convex set isometric to $\Sigma/\Pi\simeq M/G$ with
  interior $c=\Sigma_\circ(p)$ isometric to $M_\circ^*$.
  Furthermore, the isotropy group along the interior of a chamber
  face $F_i$
  is constant equal to $\K_i$.
\end{prop}

We will see in fact that  a simply connected polar $\G$ manifold
   is Coxeter polar. This also follows from work of
  Alexandrino and T\"{o}ben on singular polar foliations (see \cite{AT} (for $\G$ compact) and \cite{A}).

Note that by the usual properties of a symmetric space, an
s-representations is Coxeter polar. For a general representation we
have:

\begin{lem}[Coxeter Isotropy]\label{linearisotropy}
Let $(V,\L)$ be a polar representation with chambers $C,C_s$, and
principal isotropy group $\H$. Then

\begin{itemize}
\item[(a)] $C_s$ is
the (Weyl) chamber of the Coxeter polar representation
$(V,\L_0)$.
\item[(b)] If $(V,\L)$ is Coxeter polar, then $\L$  is generated by the
isotropy groups $\K_i$ of the faces of $C (=C_s)$. Furthermore,
the representation   is determined, up to linear equivalence, by
$\L , \H$ and $\dim V$.
\item[(c)] $(V,\L)$ is Coxeter polar if and only if   it is orbit equivalent to
$(V,\L_0)$, i.e., $\L$ is generated by $\L_0$ and $\H$.
 \end{itemize}
\end{lem}

\begin{proof}
(a) Clearly an orbit $\L v$ is singular if and only if the orbit
$\L_0 v$ is. Thus for both, the group $\Rho_s$ is the same, and
hence have the same fundamental domain. But  $\L_0$ is orbit
equivalent to an s-representation and hence Coxeter polar.

 (b) We start by showing that $\L$  is generated by the face isotropy groups
 $\K_i$. Notice that, since there are no exceptional orbits, all
 isotropy groups on the interior of the face are equal to $\K_i$.
The claim that $\K_i$ generate $\L$ is proved for general polar
actions on simply connected positively curved manifolds in
\cite{FGT}, and we indicate a proof in our context. For this, fix $C$
and $g\in\L$, and let $\L^*$ be the group
generated by $\K_i$. Using Wilking's dual foliation theorem
\cite{Wi}, applied to the orbit foliation of $\L$ on the unit sphere
 $\Sph^\ell\subset V$, one obtains (from a suitable piecewise horizontal path from a point in $C$ to an point in $gC$) a sequence, $C= C_0, C_1, \ldots, C_n = gC$, where each $C_i$ is a chamber in some section, and consecutive $C_i$'s have a wall in common. It follows that  $C_{i+1} = k_i C_i$ for some $k_i$ in the isotropy group of that wall. From this description one also sees that each $k_i \in \L^*$, and hence $gC = k_1\cdot \dots \cdot k_nC$, with $k_1\cdot \dots \cdot k_n \in \L^*$. Since the action is polar, i.e.,
$\Pi_C$ is trivial, it follows that $g = k_1\cdot \dots \cdot k_n \in \L^*$.

Before proving the second part of (b), we prove part (c). If
$(V,\L)$ is orbit equivalent to $(V,\L_0)$, then it is clearly
Coxeter polar. Conversely, assume that  $(V,\L)$ is Coxeter polar.
As we just saw, the face isotropy $\K_i$ generate $\L$. Furthermore,
$\K_i/\H$ is connected since it is a sphere of positive dimension
due to the fact that $\K_i$ is almost minimal and singular. Thus
$\K_i$ is generated by $\K_i \cap \L_0$ and $\H$, which implies that
$\L_0$ and $ \H$ generate $ \L$. In particular, $\L$ and $\L_0$ have
the same principal orbits and hence are orbit equivalent.

To finish, we prove the second claim in part (b). For this,
 observe that from the classification of symmetric spaces (see ,e.g.,
\cite{Lo}) it follows that two irreducible s-representations by the
same (connected) Lie group $\S$ are linearly equivalent if they have
the same dimension. The same then clearly holds if the
s-representation is reducible. If $\L$ is connected, recall that
$\L$ is orbit equivalent to an s-representation by some Lie group
$\S$ and that in fact $\L\subset\S$. Furthermore, $\S$ is uniquely
determined by $\L$, since otherwise the two s-representations would
be orbit equivalent. This determines the representation of $\L$ as
well. If $\L$ is not connected, we just saw that it is generated by
$\L_0$ and $\H$ and that the representation of $\L_0$ is determined
by $\dim V$. As for the action of $\H$ on $V$, it acts trivially on
the section, and orthogonal to it as the isotropy representation of
a principal orbit $\L/\H$. Thus the representation of $\L$ is
determined as well.
\end{proof}

\begin{rem}
If $(V,\L)$ is polar, $\L$ lies in the normalizer of $\L_0$ in
$\O(V)$. For an irreducible s-representation, one finds a list of
the finite group $\N_{\O(V)}(\L_0)/\L_0$ in \cite{Lo} Table 10 (in
fact $\N_{\O(V)}(\L_0)$ is precisely the stabilizer of the full
isometry group of the simply connected symmetric space corresponding
to the s-representation). To see which extensions are Coxeter polar,
one needs to determine if a new component has a representative which
acts trivially on the section of $\L_0$. For example, for the
adjoint representation of a compact Lie group $\K$,
$\N_{\O(V)}(\L_0)/\L_0\simeq \Aut(\fk)/\Int(\fk)\cup\{-\Id\}$ and
one easily sees that the only extension which is possibly Coxeter
polar is the one by $-\Id$. For a simple Lie group this extension is
in fact Coxeter polar in all cases but $\SU(n), \SO(4n+2)$ and
$E_6$. An orbit equivalent subgroup of an s-representation can of
course have a larger normalizer and such normalizers can contain
finite groups acting freely on the unit sphere in $V$. This shows
that part (b) of \lref{linearisotropy} does not hold for general
polar representations.
\end{rem}

 \smallskip

Finally, observe the following. The exponential map at $\sigma(p),
p\in \partial C$ is $\G_{\sigma(p)}$ equivariant and recall that
orbit types are totally geodesic. This implies that it takes faces
of the slice representation to faces of $C$ containing $p$ in their
closure, and that the reflection group of $\G_{\sigma(p)}$  consists
of reflection in these faces. Furthermore, the  generic isotropy
groups of the faces are the same as well, and the tangent cone
$T_qC$ is a fundamental chamber for the slice representation. This
implies in particular:
\smallskip
\begin{itemize}
\item If $(M,\G)$ is Coxeter polar, then all slice representations
are Coxeter polar and hence $\G_{\sigma(p)}$ is generated by the
face isotropy groups containing $p$ in their closure.
Furthermore, the tangent cone $T_qC$ is the (metric) product of
the tangent space to the stratum with a cone over a spherical
simplex.
\end{itemize}

  Motivated by this, we call a domain $C$ a {\it Coxeter
domain} if it is a manifold with boundary, stratified by totally
geodesic submanifolds, such that at all points of a $k$ dimensional
stratum, the tangent cone  is the product of $\R^k$ with a cone over
a spherical simplex. Notice  that for any polar action
 $C'$ and $C_s'$ are Coxeter domains.

\bigskip

%%%%%%%%%%%%%%%%%%%%%%%%%%%%%%%%%%%%%%%%%
%%%%%%%%%%%%%%%%%%%%%%%%%%%%%%%%%%%%%%%%%
%%%%%%%%%                      Reconstruction for Coxeter Polar manifolds                %%%%%%%%%%%%%
%%%%%%%%%%%%%%%%%%%%%%%%%%%%%%%%%%%%%%%%%
%%%%%%%%%%%%%%%%%%%%%%%%%%%%%%%%%%%%%%%%%

\bigskip

\section{Reconstruction for  Polar manifolds}

\bigskip

In the previous section we saw how to reconstruct an exceptional
polar manifold, i.e., one for which there are no singular orbits, or
equivalently $\Rho_s = \{1\}$, from $\G/\H$, a section $\Sigma$ and
 the associated polar group $\Pi = \Rho_e \cdot \Pi_C \subset \N(\H)/\H$, where $\H$ is the principal
  isotropy group of regular points in the section.

\smallskip

\medskip
\begin{center}
\emph{Reconstruction for Coxeter Polar manifolds}
\end{center}
\medskip

We will now discuss the reconstruction problem for  polar manifolds
and start with the simpler and most important case of Coxeter polar
manifolds. As we shall see, in this case one only needs to know the
isotropy groups and their slice representations along a
 chamber $C$.

 \smallskip

Recall that for a Coxeter polar manifold a chamber $C$ is a
 fundamental domain (in fact a Coxeter domain) for the polar group $\Pi=\Rho (= \Rho_s)$ and  $\pi\circ\sigma\colon C \to M/\G$ is an
 isometry.
 The chamber faces $F_i, i \in I$ of $C$ are uniquely identified with faces of
 $M^*$ and orbit types of $\Pi $ are identified with orbit types of $\G$. The interior of each face $F_i$ is an
 orbit space stratum which by abuse of notation we again denote by $F_i$.
  The corresponding isotropy group, which is constant,  will be
 denoted by $\K_i$. Similarly, the interior of intersections $F_{i_1} \cap F_{i_2} \cap  \ldots \cap F_{i_{\ell}}$
 (possibly empty) is a stratum denoted by $F_{i_1 ,i_2 \ldots i_{\ell}}$ with isotropy group (constant)
 denoted by
 $\K_{i_1, i_2 \ldots i_{\ell}}$.  Notice that this is true even if the intersection has more than
  one component
 since the slice representations are Coxeter polar and by \lref{linearisotropy} their face
 isotropy groups
  $K_{i_1},\dots,\K_{i_\ell}$
 generate $\K_{i_1, i_2 \ldots i_{\ell}}$.

We thus can regard the chamber $C$, or equivalently $M^*$, as being
\emph{marked} with an associated partially ordered \emph{group
graph}  where the \emph{vertices} are in one to one
  correspondence with the isotropy groups $\K_{i_1, i_2 \ldots i_{\ell}}$ along  the (non empty) strata, i.e.

\begin{equation}\label{vertices}
\bullet :   F_{i_1, i_2 \ldots i_{\ell}}  \longleftrightarrow   \K_{i_1, i_2 \ldots i_{\ell}}
\end{equation}

Moreover, a vertex $\K_{i_1 i_2 \ldots i_{\ell}}$ of ``cardinality"
$\ell$ is joined by an \emph{arrow} to each vertex $\K_{j_1 j_2
\ldots j_{\ell + 1}}$ of cardinality $\ell + 1$ containing ${i_1 i_2
\ldots i_{\ell}}$ as a subset and to no other vertices. In other
words

\begin{equation}\label{edges}
\bullet: \K_{i_1,\ldots \hat{ i_j} \ldots  i_{\ell + 1}} \longrightarrow
\K_{i_1, i_2 \ldots i_{\ell + 1}}, \quad j = 1, \ldots, \ell+1.
\end{equation}

 For isotropy groups this means that the closure of the strata for $F_{i_1\ldots \hat{ i_j} \ldots  i_{\ell + 1}}$
 contains the strata for $ F_{i_1, i_2 \ldots i_{\ell + 1}}$ and there are no strata in between.

\smallskip

We denote this group graph, or more precisely the marking of $C$, by
$\G(C)$. In \cite{GZ1} the group graph was called the group diagram
for a cohomogeneity one manifold with orbit space an interval and
denoted by $\H\subset \{\Km,\Kp\}\subset \G$. Here $\Kpm$  are
simply the markings of the endpoints.

\smallskip

Note that for a general group action one can consider a
corresponding graph where, however, the
 isotropy groups corresponding to components of orbit space strata need to be replaced by their
  conjugacy classes.

\smallskip

  There is a natural length metric on the graph $\G(C)$, in which each of the individual
   arrows has length 1. We will say that a vertex, $\K$ is at \emph{depth} $d$ if the distance from
   $\K$ to the principal vertex $\H$ is $d$. The depth of $\G(C)$ is defined to be the largest depth of all its vertices.
    Furthermore, the subgraph of the group graph for $\G$ consisting of all vertices
and arrows eventually
  ending at $\K$ will be called the  {\it history} of $\K$.

  \smallskip

In the case of a Coxeter polar action we  have the following
\emph{compatibility relations} induced by the exponential map (see
Section 2).

\begin{itemize}
\item For each vertex $F_{i_1 i_2 \ldots i_{\ell}}$, the slice
representation  is Coxeter polar with group graph the history of
$ \K_{i_1 i_2 \ldots i_{\ell}}$. The tangent cone of $C$ along a
point in the strata
 is a chamber of the slice representation
whose marking is induced by the marking on $C$.
\end{itemize}

 Also notice that the pull back of the metric on an $\e$
ball $C(q,\e)\subset C$ gives rise to a metric on an $\e$ ball in
the fundamental chamber of the slice representation. This metric, if
extended via the local polar group to a metric on $T_q\Sigma$, is
smooth near the origin.
 \smallskip

 Motivated by these properties
we define a \emph{Coxeter polar data}
 $D = (C, \G(C))$ as follows:

\begin{itemize}
\item
A  Coxeter domain $C$, i.e., a smooth Riemannian manifold with
boundary stratified by totally geodesic submanifolds.
\item
A \emph{partially ordered group graph} $\G(C)$ marking the
strata of $C$ and satisfying \eqref{vertices},
\eqref{edges}.
\item For each vertex $\F_{i_1, i_2 \ldots i_{\ell}}$ with isotropy $\K_{i_1, i_2 \ldots i_{\ell}}$ there exists a
  representation of $\K_{i_1, i_2 \ldots i_{\ell}}$ on a vector
space $V$ which is Coxeter polar and whose group graph is the
history of $ \K_{i_1, i_2 \ldots i_{\ell}}$.
\item There exists a smooth Riemannian metric $g_V$ on a small ball $B_\e(0)\subset V$ invariant under the polar group of the representation
 such that for all $q\in \F_{i_1, i_2 \ldots i_{\ell}} $ the
tangent cone  $T_qC$ is isometric to  the restriction of $g_V$
to a chamber of the representation $V$.
\end{itemize}
By \lref{linearisotropy}, the representation $V$ is in fact
determined by $\K_{i_1, i_2 \ldots i_{\ell}}$ and its history up to
linear isomorphism.

\smallskip

 We will now see how to
reconstruct the $\G$ manifold $M$ from these data, and how these
data by themselves defines a polar manifold. Note that the latter,
by definition involves constructing a (complete)
 Riemannian metric and a section for the action, restricting to the given metric on  $C$.

\begin{thm}[Coxeter Reconstruction]\label{coxeterreconstruction}
Any set of smooth Coxeter polar data $D = (C,\G(C))$  determines a
Coxeter polar $\G$ manifold $M(D)$ with orbit space $C$. Moreover,
if $M$ is a Coxeter polar manifold with data $D$ then $M(D)$ is
equivariantly diffeomorphic to $M$.
\end{thm}

\begin{proof}
We start with polar data $D = (C,\G(C))$ with a compatibly
smooth Riemannian metric on $C$ and will construct $M(D)$ as the
union of ``tubular neighborhoods" of the ``orbits" $\G/\K$,
analogous to the slice theorem. Start with a vertex $\K_{i_1 \dots
i_{\ell}}$
 and a
point $p \in F_{i_1 \dots i_{\ell}}$.
 By the compatibility condition, associated to $\K = \K_{i_1 \dots  i_{\ell}}$ we have a linear Coxeter polar
representation $V_{\K}$ with section $\Sigma_\K$ and a smooth metric
on $B_\e(0)\subset\Sigma_\K$ invariant under $\Pi_\K$ which on the
fundamental domain  restricts to the metric on the tangent cone
$T_qC$. \lref{linearisotropy} implies that this representation is
uniquely determined by the history of $\K$.

We can now use the metric extension \tref{metrics} to find a  smooth
$\K$-invariant metric on the $\e$ ball $B^*_\epsilon(0)\subset
V_{\K}$ which restricts to the metric on $B_\epsilon(0)$, and such
that $\Sigma_\K$ is  a section. Notice that although the metric on
$B_\e$ is not complete, the extension \tref{metrics} is  a local
statement about smoothness near a singular orbit and thus can be
applied in our situation as well (see the proof in \cite{Me}).

We now consider the associated homogeneous vector bundle $\G
\times_{\K} V_{\K}$ or more precisely the disc bundle $\G
\times_{\K} B_\epsilon^*$. We endow $G\times B_\epsilon^*$ with a
product metric of a biinvariant metric on $\G$ and the above $\K$
invariant metric on $B_\epsilon^*$. Under the Riemannian submersion
$\pi\colon\G \times B_\epsilon^*\to \G \times_{\K} B_\epsilon^*$
this induces a metric on the tube $M(p,\epsilon):=\G \times_{\K}
B_\epsilon^*$ and $G$ acts isometrically via left multiplication in
the first coordinate. Since the orthogonal complement of a $\G$
orbit in $M(p,\epsilon)$ is the orthogonal complement of a $\K$
orbit in $B_\epsilon^*$, $\{e\}\times B_\epsilon$  is horizontal in
the Riemannian submersion $\pi$. Thus the action of $\G$ on
$M(p,\e)$ is polar with local section $\pi(\{e\}\times B_\epsilon)$
and fundamental domain isometric to $C(p,\epsilon)\subset C$.

Altogether, we have constructed for any $\K \in \G(C)$ and any point
$p$ in the $\K$ strata
 a Coxeter polar manifold
$M(p,\epsilon)$  with  chamber $C(p,\epsilon)$.
 Moreover, if for some $\L \in \G(C)$ and  $q \in C(p,\epsilon)$ is also in the $\L$ strata,
  the orbit $\G q$
  is canonically identified with $\G/\L$ with $q$ corresponding to $[\L] \in \G/\L$. For two points
  $p_i$ with non-empty intersection  $C(p_1,\epsilon_1) \cap C(p_2,\epsilon_2) $ this provides a
  canonical equivariant identification of the corresponding orbits in
  $M(p_i,\epsilon_i)$. Thus we obtain a $\G$ manifold $M(p_1,\epsilon_1) \cup M(p_2,\epsilon_2)$
  with section $C(p_1,\epsilon_1) \cup C(p_2,\epsilon_2)$.
 Clearly then, using a
 partition of unity associated with a cover $\{ C(p,\epsilon)\}$ produces a smooth $\G$
manifold $M(D)$. Since the orbits by construction are
 (closed) embedded sub manifolds and the orbit space is complete,
 so is $M(D)$. Notice that a section
 $\Sigma$, and the polar group $\Pi$
with its action $\Sigma$, is ``developed" via repeated reflection.
To see that the metric on the section  is complete, we can, e.g.,
apply \tref{integrable} since the horizontal distribution on
$M(D)_\circ$ is integrable by construction.

To see that the action of $\G$ on $M(D)$ is Coxeter polar, first
observe that it has no exceptional orbits since by assumption, the
slice representations have this property. Furthermore, $M(D)/\G=C$,
and hence $C$ is a fundamental domain for the polar group $\Pi$
acting on $\Sigma$. By construction $\Pi$ contains all the
 reflections in the (chamber) faces of $C$. Along each stratum, these
 generate by assumption (the local data are Coxeter) a reflection group of
 the normal sphere to the stratum and hence act freely and transitively on
  chambers in the normal sphere. If $\Rho$ is larger than the group generated
   by the reflections in the faces of $C$, then a chamber for $\Rho$ would be
    properly contained in $C$ and hence $C$ could not be the orbit space.
    Similarly $\Pi_C$ must be trivial since again $C$ is the orbit space
    $C/\Pi_C$.

\smallskip

 Conversely, if $D$ is the Coxeter polar data of a given manifold
 $M$, the construction above,
 and the tubular neighborhood theorem for the $\G$ action on $M$,
 clearly gives a canonical
  identification of $M$ with $M(D)$ which is a $\G$ equivariant diffeomorphism.
\end{proof}

\begin{rem*} (a) Notice that if we start with a section $\Sigma$ on which $\Pi$
acts isometrically with $\Pi_C=\{e\}$, and  mark the fundamental
domain $C$ with some group graph, the new section in $M(D)$
constructed above may not be equal to $\Sigma$. However, all
sections are covered by the universal section $\Sigma'$ associated
to the Coxeter matrix $\M$  defined in terms of the faces of $C$ and
their angles between them (see Section 2).
 In particular, the section is unique if
$\Pi=\M$. This happens, e.g., for any Coxeter polar manifold with a
fixed point since the polar group is clearly the local polar group
at the fixed point.

Notice also that the group graph
$\G(C)$ is already completely determined by the (almost minimal)
isotropy groups $\K_i$ since they generate all other vertex groups.
Also $\Pi$ is completely determined by the group graph since it is
generated by the reflections in the faces $F_i$ of $C$. As elements
in $\N(\H)/\H$, they are  the unique involutions in $(\N(\H)\cap
\K_i)/\H$ (recall that $\K_i/\H$ is a sphere). In particular, $M$ is
compact if and only if $C$ is compact and $\Pi$ is finite. On the
other hand, one can change the metric on $C$, or the conjugacy
classes of the vertex groups in $\G(C)$, without changing the
equivariant diffeomorphism type of $M(D)$. But such a change can
change the topology of $\Sigma$ and the polar group $\Pi$. It is an
interesting question if, via such a change, one can make the section
embedded, and if the manifold and group is compact, one can make the
section compact as well. This is easily seen to be true for a
cohomogeneity one action.

 (b) We can of course simply enlarge
 the group $\G\subset \L$ to obtain new polar data $D'(C,L(C))$ with the same
 isotropy groups. The new manifold can be regarded as $M\times_{\G}
 \L$. Conversely, the polar group action is called not primitive (and primitive otherwise) if there exists
 a $\G$-equivariant map $M(C,\G(C))\to \G/\L$ for some subgroup
 $\L\subset\G$. The fiber over the coset $[\L]\in \G/\L$ is then
 polar with data $(C,\L(C)$, i.e., all $\G$ group assignments are
 contained in $\L$. The section and polar group is the same for both
 polar actions.

(c) If we denote the inverse of $\pi\circ\sigma_{|C}$ by $s\colon
M^*\to
 C\subset\Sigma$, then $\sigma\circ s$ provides a ``splitting" for $\pi\colon M\to
 M^*$. In particular, the orbit type $\N(\K)/\K$ principle bundles $
 M^\K\cap M_{(\K)}\to M^*_{(\K)}$ have a canonical section. See also \cite{HH},
 where they examined group actions with such a splitting
  in the topological category, and a reconstruction of the
  topological space with an action of $\G$ as well. But in this
  category there is no compatibility condition.

(d) We point out that under very special assumptions, this Theorem
was claimed (without proof) in the cohomogeneity two case in
\cite{AA2}.
\end{rem*}

\medskip
\begin{center}
\emph{Reconstruction for General Polar manifolds}
\end{center}
\medskip

We will now address the general case, where both
$\Rho_s$ and $\Pi_C$ are
 non-trivial (recall that  $\Rho_C$ could be non-trivial as well).  Rather than considering $C$
 we will consider $C_s$ or better its Coxeter
 completion $C_s'$ together with the (induced) action by
 $\Pi_{C_s}$. We point out that replacing $C$ by $C_s$ in our
 reconstruction is linked to the properties of linear polar
 actions in \lref{linearisotropy}.
Note that $C_s$ is tiled by copies of $C$ under exceptional
reflections and   $\Pi_C\subset\Pi_{C_s}$, i.e. $\Pi_{C_s}$ is
nontrivial as well.
 Using the reconstruction for Coxeter
  polar manifolds we will now construct a cover which is  Coxeter polar.

\begin{thm}[Coxeter Covers]\label{cover}  Given a polar manifold $M(C,\G(C))$ with polar
group $\Pi$ and completion $C_s'$, there exists a Coxeter polar
manifold $M(C_s',\G'(C_s')$ with polar group $\Rho_s$ and a free
action by $\Pi_{C_s}'$, commuting with $\G$, whose quotient is our
$\G$ manifold $M$.
\end{thm}
\begin{proof}
Recall that $C_s'$ is indeed a Coxeter domain. We  now need to
assign a group graph to $C_s'$ satisfying the compatibility
conditions in \tref{coxeterreconstruction} . Recall that to each
face $F_i$
 we associate the generic isotropy group $\K_i$ of the open chamber face.
 Consider a non-empty
 intersection of faces
 $F_{i_1 ,i_2 \ldots ,i_{\ell}}=F_{i_1} \cap F_{i_2} \cap  \ldots \cap F_{i_{\ell}}$
 and pick a generic point $p$ in its interior. The isotropy
 $G_{\sigma(p)}$ acts on the slice at $p$  and its generic isotropy groups
 are (via the exponential map) the isotropy groups $\K_{i_1}, \ldots,
 \K_{i_{\ell}}$. Let $\K_{i_1 \ldots i_\ell} \subset G_{\sigma(p)}$
 be the subgroup generated by these generic isotropy groups, and
 restrict the slice representation to $\K_{i_1 \ldots i_\ell} $ as
 well. By \lref{linearisotropy}, this representation of $\K_{i_1 \ldots i_\ell}
 $ is now Coxeter polar.
 All together, these isotropy groups
define the Coxeter data $D'=(C_s', \G(C_s'))$ for a polar $\G$
manifold $M'=M(D')$ with  orbit space $C_s'$. One easily checks that
the compatibility conditions are all satisfied. Note also that by
construction, the polar group of $M'$ is the reflection group
$\Rho_s$ for
 $M$, a normal subgroup of $\Pi=\Rho_s \cdot \Pi_{C_s}  \subset
 \N(\H)/\H$ with
    $\Rho_s \cap \Pi_{C_s} = {(\Rho_s)}_{C_s}$.

\smallskip

Now consider the stabilizer group $\Gamma : = \Pi_{C_s} \subset \Pi$
of $C_s$ and its action on $C_s'$ as well as on orbits or strata of
orbits: Let $x, y$  lie in two strata of $C_s'$ (possibly the same)
 with generic isotropy $\K_{i_1
\ldots i_\ell}$ and $\K_{j_1 \ldots j_\ell}$. If  $y=\gamma x$ with
$ \gamma\in\Gamma$, then we of course have $\G_y=\gamma
G_x\gamma^{-1}$, i.e  $\gamma \K_{i_1 \ldots i_\ell}\gamma^{-1}
=\K_{j_1 \ldots j_\ell}$. We  therefore have a natural
identification of the orbit through $x$ with the orbit through $y$
via $g\K_{i_1 \ldots i_\ell}\to g\gamma^{-1} \K_{j_1 \ldots
j_\ell}$. This defines an  action of $\Gamma$ on $M'$ commuting with
$\G$. We can choose a metric on $M'$ as in the proof of
\tref{coxeterreconstruction} such that $\Gamma$ acts by isometries
and the metric on $C_s$ is unchanged.  Moreover this action is free.
Indeed, if $x=\gamma x$ then $\gamma\in \N(\K_{i_1 \ldots i_\ell})/
\K_{i_1 \ldots i_\ell}$. But recall that in general for a
homogeneous space $\G/\L$
 the action of $\N(\L)/\L$ on $\G/\L$  on the right is free.

We now have a free action of $\Gamma$ on $M'$ and an induced action
by $\G$ on the quotient $N = M'/\Gamma$. We claim that $(N,\G)$ is
equivariantly diffeomorphic to the action on the given $(M,\G)$.
They both have the same fundamental domain $C_s$, and completion
$C_s'$, by construction. Also recall that the isotropy in $M$ at
$p\in F_{i_1 ,i_2 \ldots ,i_{\ell}}$ is the generic isotropy
$\K_{i_1 ,i_2 \ldots ,i_{\ell}}$ extended by the isotropy of
$\Gamma$ at $p$. But this is also the isotropy of the $\G$ action on
$N$ and we can hence define the diffeomorphism orbit wise.
\end{proof}

\begin{rem}
(a) This also reproves in a concrete way the results of Alexandrino
and T\"oben that polar actions on simply connected manifolds by a
connected group have no exceptional orbits, and that such an action
is Coxeter polar. Indeed, the existence of exceptional orbits means
that $\Gamma = \Pi_{C_s}$ is non trivial, and hence one gets a cover
 (which is connected since $\G$ is connected).

(b) One can interpret this even when $\Rho_s$ is trivial, in which
case the ``Coxeter domain" $C_s' = C_s = \Sigma$, thus covering
exceptional polar actions.
\end{rem}

\bigskip

The proof of \tref{cover} also shows how to take quotients of
Coxeter polar manifolds.

\begin{thm}[Coxeter Quotients]\label{quotient}
Let $(M,\G)=M(C,\G(C))$ be a Coxeter polar manifold with  polar
group $\Pi$. Let $\Gamma\subset \N(\H)/\H$ be a subgroup that
normalizes $\Pi$ and acts on $(C,\G(C))$, i.e., $\Gamma$ acts on $C$
 isometrically preserving
strata and conjugates the corresponding group assignment.
 Then $\Gamma$
acts isometrically and freely on $M$  commuting with the action of
$\G$ and the quotient $(M/\Gamma,\G)$ is polar with polar group
$\Pi\cdot\Gamma$.
\end{thm}

\begin{rem}This also encodes that if a subgroup of $\Gamma$ preserves a stratum, it
must lie in the normalizer of the assigned  group. The section in
$M/\Gamma$ is of course simply the image under the cover and hence
$\Gamma\cap\Pi$ is the stabilizer group of a chamber in $M/\Gamma$
and the section of $M/\Gamma$ has $| \Pi/\Gamma \cap\Pi|$ chambers.
We also  allow the possibility that $\Gamma\subset\Pi$ is a proper
subgroup.
\end{rem}

%%%%%%%%%%%%%%%%%%%%%%%%%%%%%%%%%%%%%%%%%
%%%%%%%%%%%%%%%%%%%%%%%%%%%%%%%%%%%%%%%%%
%%%%%%%%%                      Examples of polar actions                %%%%%%%%%%%%%
%%%%%%%%%%%%%%%%%%%%%%%%%%%%%%%%%%%%%%%%%
%%%%%%%%%%%%%%%%%%%%%%%%%%%%%%%%%%%%%%%%%

\bigskip

\section{Constructions, Applications and Examples}

\bigskip

Although polar actions may seem rather special and rigid, the
construction from data provided in the previous  section is amenable
to various operations leading to a somewhat surprising flexibility
and a wealth of examples. In addition to the trivial process of
taking \emph{products}, these operations include general
\emph{surgery type constructions}, as well as the construction of
\emph{principal bundles}.  When
combined with the
 flexibility in the choice of metrics this can be used to derive topological conclusion by means of geometric arguments.
 \medskip
\begin{center}
\emph{Cutting and Pasting Operations}
\end{center}
\medskip

Consider a polar $\G$ manifold $M$ with an invariant hyper surface
$E$ corresponding to a $\Pi$ invariant separating hyper surface
$\Delta \subset \Sigma$ missing $0$ dimensional $\Pi$ strata.
Accordingly we have decompositions $M = M_1 \cup M_2$, $\Sigma =
\Sigma_1 \cup \Sigma_2$ and $M^* = M^*_1 \cup M^*_2$, where
 $M_i$ are polar $\G$ manifolds with common boundary $E$, and $\Sigma_i$ are $\Pi$ manifolds with common
  boundary $\Delta$. Here, typically $\Pi$ does not act effectively on $\Delta$ and $\Pi/ker$ is the polar
   group for the section $\Delta$ in $E$. Note also that $\Delta^*$ determines $\Delta$ via invariance and is automatically perpendicular to the strata it meets. We will also refer to such a $\Delta^*$ as a separating hyper surface in $M^* = \Sigma^*$.

Now suppose $M'$ is another polar $\G$ manifold which is separated in the same fashion into two polar $\G$
 manifolds $M'_1$ and $M'_2$ with common boundary $E'$. If moreover there is a $\G$ equivariant isometry
 $A$ from a neighborhood of $E \subset M$  to a neighborhood of $E' \subset M'$, then clearly
 $M_1 \cup_{A} M'_2$  is a polar $\G$ manifolds as well, with section
 $\Sigma_1 \cup_{A} \Sigma'_2$.

Notice that as long as a separating hypersurface $\Delta^*$ in $M^*$ is
orthogonal to all the strata it meets, the inverse image in $M$ is a
smooth  separating hypersurface $E$. The condition that a
neighborhood of $E$ is $\G$ isometric to a neighborhood of $E'$ can
be
 achieved via our reconstruction process (including metrics)  as long as there is a $\Pi/ker$ invariant
 isometry between neighborhoods of $\Delta$ and $\Delta'$. To achieve this modulo a change of metrics
 on the sections $\Sigma$ and $\Sigma'$ it suffices to have a $\Pi/ker$ invariant isometry between
 $\Delta$ and $\Delta'$: Indeed, making the geodesic reflections in $\Delta$ and $\Delta'$ isometries
  in tubular neighborhoods makes $\Delta$ and $\Delta'$ totally geodesic. Next one can change the metrics
   to be product near $\Delta$ and $\Delta'$, in both steps keeping the actions by the polar groups isometric
   (also with respect to a partition of unity).

This in particular proves the following

\begin{thm}
Let $M$ and $M'$ be Coxeter polar $\G$ manifolds with orbit spaces
$C$ and $C'$.  Let $\Delta$, respectively ${\Delta'}$ be codimension
1 separating hypersurfaces  in $C$ and $C'$ orthogonal to the strata
they meet.
 Suppose  there is an isometry $A: \Delta \to {\Delta'}$ such that corresponding isotropy data are the same.
 Then $A$ induces a $\G$-equivariant diffeomorphism between the codimension one sub manifolds of $M$ and $M'$
 corresponding to $\Delta$
 and ${\Delta'}$. In particular,
  $M_1 \cup M'_{2}$ has the structure of a polar $\G$ manifold.
\end{thm}

\begin{rem*}
(a) The separating hyper surfaces need not meet any lower strata.
This way for example one can glue together  two copies of the
complement of a tubular neighborhood of a principal orbit (see also
\cite{AT}). In fact, this is possible for any orbit as long as the
slice representations are equivalent. Another special case is the
connected sum along fixed points if the isotropy representations at
the fixed points are equivalent.

(b) $\G$ equivariant standard surgery can be performed in the
following case. Suppose $\S^k \subset M^n$ is $\G$ invariant sphere
in $M$ with trivial normal bundle. Assume $\G \colon \R^{k+1} \oplus
\R^{n-k} \to \R^{k+1} \oplus \R^{n-k}$ is a reducible polar
representation, such that when restricted to a tubular neighborhood
of $\Sph^k \subset \Sph^n$ is equivalent to the restriction of the
$\G$ action on $M$ to a tubular neighborhood of $\S^k$.
 Then the manifold $N$ obtained from $M$ by replacing $\Sph^k \times \Disc^{n-k}$ with $\Disc^{k+1} \times \Sph^{n-k-1}$
 is a polar $\G$ manifold.
\end{rem*}

Here are some interesting concrete examples applying this type of operation:

\begin{example}
(a) Consider the standard cohomogeneity one action by $\SO(n)$ on
$\Sph^n$ with two fixed points, section a circle and polar group
$\Z_2$. The cohomogeneity two product action by $\G = \SO(n) \times
\SO(m)$ on $\Sph^n \times \Sph^m$ is polar with section a torus and
polar group $\Z_2 \times \Z_2$. By performing a connected sum at
 fixed points of $\G$ it follows
from the above that  $M = \Sph^n \times \Sph^m \# \Sph^n \times
\Sph^m$ admits a polar $\G$ action with section $\Sph^1\times\Sph^1
\#\Sph^1\times\Sph^1$ and orbit space a right angled hexagon. In
particular M admits a polar metric with section of constant negative
curvature. We can of course also take further connected sums and
obtain a section of any given higher genus.

(b) Consider the standard polar action by $\T^n$ on $\CP^n$ (and the
one where the action is reversed via the inverse) with $n+1$ fixed
points and section $\RP^n$.
 Thus $\CP^n \#
{\pm} \CP^n$ admits a polar action and in fact admits
 a non negatively curved invariant metric (using the Cheeger construction \cite{Ch}) with section $\RP^n \# \RP^n$.
 Taking  further connected sums $M = \CP^n \# \dots \# \CP^n$ yields a polar $\T^n$ manifold with section
 $\RP^n \#\ldots \# \RP^n$. If $n=2$, we can equip $M$ with a metric such that the section has
 constant curvature $0$ or $-1$.

 (c) In the first example (and
similarly in the second) we can do a surgery along a principal
orbit on two copies of the action. The section will be the double of $\Sph^1\times\Sph^1$ minus
four discs, one at the center of all four chamber squares of
$\Sph^1\times\Sph^1$. The polar group remains $\Z_2 \times \Z_2$.

(d) Recall that a symmetric space
$M=\G/\H$ has a natural polar action by $\H$ with a fixed point.
Thus $M\# M$ has a polar action by $\H$ as well.
\end{example}

In general using such constructions with the basic sources coming from cohomogeneity one manifolds and \emph{model examples} coming from polar actions on symmetric spaces (as above) leads to a wealth of interesting examples.

\smallskip

Here are some further (model) examples illustrating our data:

\begin{example}
The case where the section is $2$-dimensional is particularly
simple, but already gives rise to many interesting examples. This
case is also easier geometrically since by the uniformization
theorem, and the metric extension theorem, we can assume that
$\Sigma$ has constant curvature and that $\Pi$ acts by isometries.
Thus for a Coxeter polar action we can simply start with a smooth
metric on a domain $C$ of constant curvature such that the boundary
$\partial C$ is  a geodesic polygon.  Recall also that the angle
between two geodesics in $\partial C$ is one of $\pi/\ell,\ $ with $
\ell=2,3,4,6$ and the corresponding local Weyl group is the dihedral
group $D_\ell$. The sign of the curvature is determined by applying
the Gauss Bonnet theorem to $C$, in particular, the section is
typically hyperbolic.

If $\K_i$ are the isotropy groups of the sides, $\L_i$ the isotropy
groups at the vertices, and $\H$ the principle isotropy group, then
the compatibility condition simply says that $\K_i/\H$ are spheres,
and that $\L_i$ defines a cohomogeneity one action on another
sphere, with group diagram given by $\L_i$, the two adjacent sides,
and the principle isotropy group $\H$. Notice also that the
smoothness condition (i.e. extension to a smooth metric in a
neighborhood of $p\in\partial C$ invariant under the local polar
group) is now automatic since the metric has constant curvature. Of
course the angle at the vertex is also determined by the vertex
groups $\L_i$. The marking by the groups also determines the polar
group $\Pi\subset \N(\H)/\H$ since it is generated by the
reflections in the sides. Thus if $C$ is compact and $\Pi$ finite,
the isotropy groups and the order $|\Pi|$ also determine the genus
of the (compact) section, again by the Gauss Bonnet theorem.

 We can of course also change the
topology of $\Sigma$ on the interior, and hence the topology of $M$,
via connected sums with any surface and keep the same isotropy
groups. But it also becomes clear that for example for a domain $C$
with 3 vertices, the compatibility conditions makes it quite
delicate for the group diagram to be consistent since the sides must
be the common isotropy groups of three cohomogeneity one actions on
spheres (but only after it has been made effective). See e.g.
\cite{GWZ}  for a list of cohomogeneity one actions on spheres,
together with their Weyl group. Notice that in the case when the
slice representation is reducible, i.e. the angle between the sides
is $\pi/2$ and $\Pi=D_2$, the description is contained in what is
called a generalized sum representation in \cite{GWZ}.

\smallskip

A typical example of a cohomogeneity two action on a rank one
symmetric space is given in Figure 1. The section is an $\RP^2$ and
the polar group is $\Pi= \Rho = C_3/\Z_2$ generated by the 3
reflections in the sides with the further relation coming from the
fact that the $C_3$ action on $\Sph^2$ contains the antipodal map.
The Coxeter group associated to $C$ is $\M=C_3$ and the universal
section $\Sigma'=\Sph^2$.

\begin{picture}(100,200)  %<<<2
 \small\put(220,140){\line(2,-3){80}}
  \put(220,140){\line(-2,-3){80}}
 \put(140,20){\line(1,0){160}}
 \put(210,148){$\Sp(3)$}
 \put(120,80){$\Sp(1)\Sp(2)$}
  \put(280,80){$\Sp(2)\Sp(1)$}
    \put(210,60){$\Sp(1)^3$}
     \put(85,10){$\S(\U(2)\U(4))$}
      \put(310,10){$\Sp(2)\U(2)$}
       \put(200,2){$\Sp(1)^3\S^1$}
        \put(155,28){$\pi/4$}
         \put(270,28){$\pi/2$}
          \put(212,115){$\pi/3$}
 \put(100,-23){\it
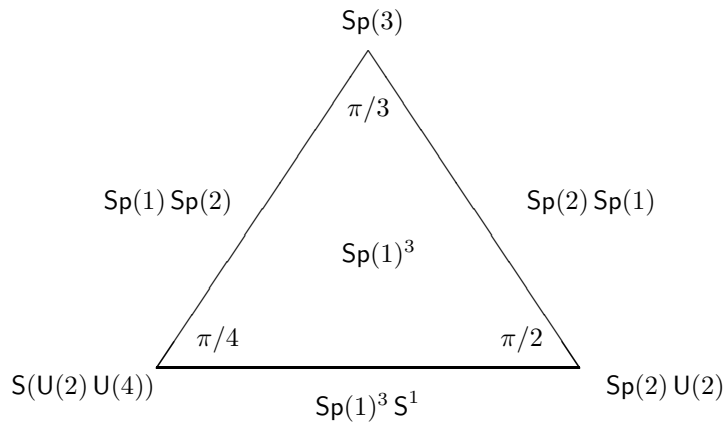
\small Figure 1. Cohomogeneity two polar action of $\SU(6)$ on
$\CP^{14}$.}
\end{picture}

\vspace{50pt}

An example of an action on a symmetric space of rank$>1$, where the
sections have to be flat, is given in Figure 2. The vertices are
fixed points and the slice representations are the well known
irreducible polar action of $\SO(3)$ on $\R^5$. $\O(2),\, \O'(2),\,
\O''(2)$ are the 3 different block embeddings of $\O(2)$ in
$\SO(3)$. The polar group is the polar group at the fixed point and
hence $\Pi=D_3$ (which in this case is all of $\N(\H)/\H$). Notice
that here there are only 2 orbits types, besides the principal one,
each of which has 3 components, and that the isotropy along the 3
sides must all be different in order for the group diagram to be
consistent. We could add an action of $\Pi_C=\Z_3\subset\Pi$ here
that rotates the vertices, which defines a polar action on a
subcover of $\SU(3)/\SO(3)$. Notice though that $\Z_2\subset\Pi$ is
not allowed since it is not normal in $\Pi$.

\begin{small}
\begin{picture}(100,200)  %<<<2
 \small\put(220,140){\line(2,-3){80}}
  \put(220,140){\line(-2,-3){80}}
 \put(140,20){\line(1,0){160}}
 \put(210,148){$\SO(3)$}
 \put(140,80){$\O(2)$}
  \put(280,80){$\O'(2)$}
    \put(210,60){$\Z_2^2$}
     \put(105,10){$\SO(3)$}
      \put(310,10){$\SO(3)$}
       \put(200,2){$\O''(2)$}
        \put(155,28){$\pi/3$}
         \put(270,28){$\pi/3$}
          \put(212,115){$\pi/3$}
 \put(80,-23){\it
\small Figure 2. Cohomogeneity two polar action of $\SO(3)$ on
$\SU(3)/\SO(3)$.}
\end{picture}
\end{small}

\vspace{50pt}

\smallskip

In the third example, where the polar manifold is not a symmetric
space anymore, we replace in the previous example two of the vertex
groups by $\SO(4)$ and let $\G=\SO(4)$. The slice representation of
$\SO(4) $ at the 2 fixed points is given by the 8 dimensional
irreducible polar representation coming from the symmetric space
$\G_2/\SO(4)$. Care needs to be taken for the 3 embeddings of
$\O(2)$ in $\SO(4)$. To describe this, it is easier to let
$\G=\S^3\times\S^3$ act ineffectively. Then for the cohomogeneity
one action by $\S^3\times\S^3$ on $\Sph^7$ the isotropy groups are
$\H=\Delta\Q$ and $\K_-=e^{(i\theta,3i\theta)}\cdot \H, \;
\K_+=e^{(j\theta,j\theta)}\cdot \H$ up to conjugacy by an element in
$\S^3\times\S^3$ (see e.g. \cite{GWZ}). For the cohomogeneity one
action by $\S^3$ (effectively $\SO(3)$) on $\Sph^4$ they are
$\K_-=e^{i\theta}\cdot \H, \; \K_+=e^{j\theta}\cdot \H,\; \H=\Q$ up
to conjugacy, and the embedding of the third vertex group
$\SO(3)\subset\SO(4)$ corresponds to the embedding $\Delta
\S^3\subset \S^3\times\S^3$. Thus a consistent choice for the
isotropy of the sides is e.g. $ e^{(i\theta,3i\theta)}\cdot \H, \;
e^{(j\theta,j\theta)}\cdot \H$ and $e^{(k\theta,k\theta)}\cdot \H$.
The polar group $\Pi$ is now $D_6$, and since $\Sigma'=\Sigma$ one
easily sees that  the genus of the section is 2. Notice that the
section here must be orientable since $C$ is and $\Pi$ acts simply
transitively on chambers.

\smallskip

Notice though that in Figure 3 it is not possible to have one or
three fixed points since there is no compatible choice of side
groups. On the other hand, we can easily generalize this example to
a domain $C$ with arbitrarily many sides, or by adding handles in
the interior. In order for the group diagram to be consistent, the
fixed points by $\SO(4)$ must come in adjacent pairs, and the
remaining vertices have isotropy $\SO(3)$. This gives rise to a
large class of 8 dimensional polar manifolds.

\begin{picture}(100,200)  %<<<2
 \small\put(220,140){\line(2,-3){80}}
  \put(220,140){\line(-2,-3){80}}
 \put(140,20){\line(1,0){160}}
 \put(210,148){$\SO(4)$}
 \put(140,80){$\O(2)$}
  \put(280,80){$\O'(2)$}
    \put(210,60){$\Z_2^2$}
     \put(105,10){$\SO(4)$}
      \put(310,10){$\SO(3)$}
       \put(200,2){$\O''(2)$}
        \put(155,28){$\pi/6$}
         \put(270,28){$\pi/3$}
          \put(212,115){$\pi/6$}
 \put(80,-23){\it
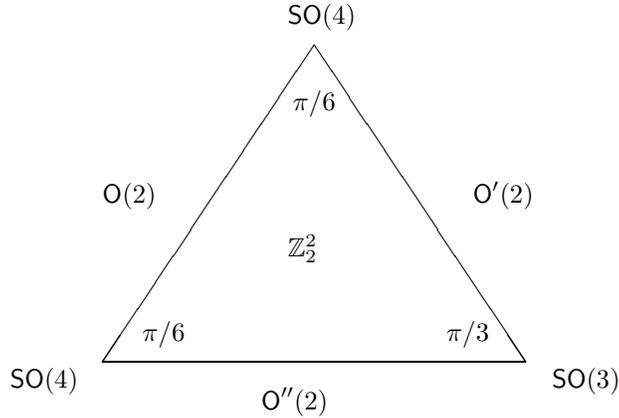
\small Figure 3. Cohomogeneity two polar action of $\SO(4)$ on
$M^8$.}
\end{picture}

\end{example}

\bigskip

\begin{example}
We now claim that every $\T^{n}$ action  on a
 compact simply connected
$(n+2)$-dimensional manifold $M$ is polar, if it  has singular
orbits. Indeed, by the theory of Orlik-Raymond \cite{OR}, every such
torus action is classified by the following data: The quotient
$M/\T^{n}$ is a 2-disk with $k$ edges and $k$ vertices. The
principal isotropy is trivial, the vertices have isotropy $\T^2$,
and the isotropy groups along the sides are $\S^1$ with some slopes,
the only condition being that $\T^2$ is the product of the two
circles assigned to any two adjacent sides. Two assignments give the
same action, if and only if they can be carried into each other by
an automorphism of $\T^n$. The slice representations at the vertices
are effective representations of $\T^2$ on $\R^{n+2}$ and hence the
angles at the vertices are $\pi/2$. But every such representation is
polar. Hence all compatibility conditions are satisfied, and we can
apply \tref{coxeterreconstruction} to obtain a polar action on
$M^{n+2}$.

In \cite{OR} it was shown that if $n=4$, the manifolds are
diffeomorphic to connected sums of $k$ copies of
$\Sph^2\times\Sph^2$ or $\pm\CP^2$, but on each such manifold there
are infinitely many distinct $\T^2$ actions, the only exceptions
being $\CP^2$ and $\CP^2\#\CP^2$ where the action is unique.

 If $n=5$     such
5 manifolds are diffeomorphic to $\Sph^5$ or connected sums of
arbitrarily many copies of $\Sph^3\times\Sph^2$, possibly with the
non-trivial $\Sph^3$ bundle over $\Sph^2$ as well (see \cite{Oh1}).
Again, each space except $\Sph^5$, admits infinitely many $\T^3$
actions.

In dimension 6, it was shown in \cite{Oh2} that such a manifold is
diffeomorphic to connected sums of arbitrarily many copies of
$\Sph^4\times\Sph^2$,\, $\Sph^3\times\Sph^3$, and possibly with the
non-trivial $\Sph^4$ bundle over $\Sph^2$ as well, again with
infinitely many actions in most cases.
\end{example}

\bigskip

\begin{center}
\emph{Bundle Constructions}
\end{center}

\bigskip

Note that the following \emph{quotient operation} holds within the
class of polar actions: If $(M,\G)$ is polar with section $\sigma$,
and $\L$ a normal subgroup of $\G$ acting freely on $M$, then
$(M/\L, \G/\L)$ is polar with section $/\L \circ \sigma$.

In the
language of
 principal bundles, the polar action of $\G/\L$ on $M/\L$ admits
 a commuting lift to the total space $M$.
 We
are interested in such {\it commuting lifts}. See \cite{GZ2} for a
discussion of  commuting lifts in the cohomogeneity one situation.
Recall also that one  obtains a larger collection of commuting lifts
if one allows the action of $\G$ on $M$ to be almost effective, i.e
$\G$ and the principle isotropy group $\H$ have a finite central
subgroup in common.

 Such bundles can be constructed via data as
follows.

\begin{thm}\label{lifts}
Let $D = (C, {\G}(C))$ be the (almost effective) data for a Coxeter
polar $\G$ manifold $M(D)$.
 Let  $\L$ be a Lie group and  choose for each $\K \in {\G}(C)$  homomorphisms
  $\phi_{\K}\colon\K\to\L$ such that ${\phi_{\K}}_{|\U}=\phi_{\U}$ for
  any $\U\in{\G}(C)$  connected in the graph $\G(C)$ to $\K$ via the inclusion $\U\subset\K$.
   Then for each $\K \in
  {\G}(C)$ we obtain embeddings of $\K$ into $\G^*=\L\times{\G}$
  via $k\to (\phi_{\K}(k),k)$. This defines a group graph
  ${\G^*}(C)$, and hence the data $D^* = (C, {\G^*}(C))$   for a Coxeter polar
$\G^*$
    manifold $ M(D^*)$. Moreover, $ M(D^*)$ is a principal $\L$ bundle over $M(D)$
    and the polar metrics on $ M(D^*)$ and $M(D)$  can be chosen such that
      the projection  is a ${\G}$ equivariant Riemannian submersion.
\end{thm}
\begin{proof}
Notice that in the group graph $\G^*(C)$ the vertex groups $\K$, as
well as their slice representations $V_\K$, are unchanged. Thus the
necessary compatibility conditions are satisfied and hence we obtain
a polar manifold $M(D^*)$. The subgroup
$L\simeq\L\times\{e\}\subset\L\times\G$ acts along the orbits
$(\L\times \G)/K$ from the left and the action is free (and
isometric) since the embedding of $\K$ in $\L\times\G$ is injective
in the second coordinate by construction (since $\L$ is normal, the
isotropy groups are simply the intersection of $\L\times\{e\}$ with
$\K$). Thus we have a principle bundle $\pi\colon M(D^*)\to
M(D^*)/\L=B$ and on the base $B$ we obtain an induced metric and
isometric action by $\G$, such that $\pi$ is a $\G$-equivariant
Riemannian submersion. The section in $M(D^*)$ is horizontal w.r.t.
 $\pi$ and hence it covers an immersed submanifold in $B$ orthogonal
to the $\G$ orbits. Thus $(B,\G)$ is polar, and since $
M(D^*)/(\L\times\G)=B/\G$ the fundamental domain is isometric to the
given metric on $C$. One easily checks (orbitwise) that the isotropy
groups of the $\G$ action on $ B$ are the same as that in $M(D)$,
i.e.  the marking of $C$ for $\G$ action on $B$ is the same as that
for $M(D)$ and hence
\tref{coxeterreconstruction} implies that $B$ is equivariantly
diffeomorphic to $M(D)$.
\end{proof}

\begin{rem*}(a) Unlike in the cohomogeneity one case, if $P\to B$ is an $\L$
principle bundle and $\G$ acts polar on $B$ and admits a lift to $P$
that commutes with $\L$, the action of $\L\times\G$ on $P$ may not
be polar.

(c) Such a principle bundle construction was carried out in the
topological category in \cite{HH}.
\end{rem*}

\begin{example}

In \cite{GM} Garcia and Kerin showed that every  $\T^2$ action on
$\Sph^2\times \Sph^2$ or $\CP^2\#\pm\CP^2$, and every $\T^3$ action
on $\Sph^3\times \Sph^2$ or the non-trivial $\Sph^3$ bundle over
$\Sph^2$,  admits an invariant metric with non-negative sectional
curvature.  The proof shows that the
manifolds can be described by taking a quotient of
$\Sph^3\times\Sph^3$ by a sub torus $\S\subset \T^4$ and the action
is the induced action by $\T^4/\S$. Since the action by $\T^4$ is
polar in the product metric, it follows that the $\T^2$ and $\T^3$
actions on these 4 and 5 manifolds admit polar metrics with
non-negative sectional curvature. Note that this is an example of
our ``quotient'' principle bundle construction. In Figure 4 we
depict the group diagram  on the $4$ dimensional manifolds (see
\cite{GM}).
\end{example}

\begin{picture}(100,200)  %<<<2
 \small
 \put(140,140){\line(1,0){120}}
\put(140,140){\line(0,-1){120}} \put(260,140){\line(0,-1){120}}
 \put(140,20){\line(1,0){120}}
 \put(130,145){$\T^2$}
 \put(125,10){$\T^2$}
 \put(265,145){$\T^2$}
 \put(268,10){$\T^2$}
  \put(110,70){$\S^1_{(0,1)}$}
   \put(270,70){$\S^1_{(k,1)}$}
    \put(190,150){$\S^1_{(1,0)}$}
     \put(190,5){$\S^1_{(1,0)}$}
     \put(148,125){$\frac{\pi}{2}$}
     \put(245,125){$\frac{\pi}{2}$}
     \put(245,30){$\frac{\pi}{2}$}
      \put(148,30){$\frac{\pi}{2}$}
 \put(60,-23){\it
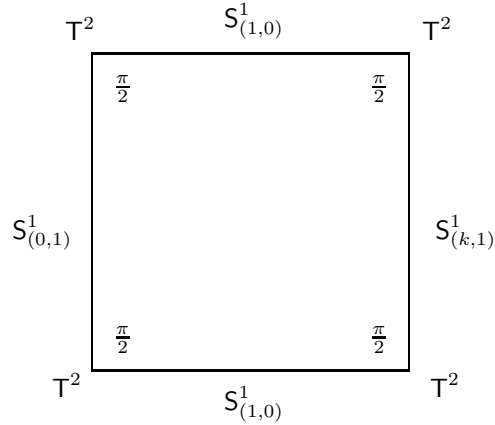
\small Figure 4. $\T^2$ polar action on  $\Sph^2\times\Sph^2$ ($k$
even) and $\CP^2\#-\CP^2$ ($k$ odd).}
\end{picture}

\vspace{30pt}

\bigskip

\bigskip

\begin{center}
\emph{Geometric and Topological Ellipticity and hyperbolicity}
\end{center}

\bigskip

Polar manifolds with a flat section are referred to as
\emph{hyperpolar} in the literature. More generally we will examine
polar manifolds with constant curvature sections, which we will
refer to as \emph{polar space forms} or more
 precisely a \emph{spherical} respectively \emph{hyperbolic} polar manifold if the section has positive
  respectively negative constant curvature.

  Recall that $M$ is called (topologically) elliptic if the Betti numbers of the loop
space grow at most polynomially for any field of coefficients. It is
well known that simply connected homogeneous spaces and biquotients
have this property. If this holds for rational coefficients, the
space is called rationally elliptic, which is also equivalent to the
condition that all but finitely many homotopy groups are finite. If
the manifold is not rationally elliptic, it is called rationally
hyperbolic. Rationally elliptic manifolds are severally restricted,
see \cite{FHT}, and \cite{GH1} in the context of non-negative
curvature.

\begin{thm}\label{elliptic}
Let $M$ be a simply connected closed polar manifold which is hyper
polar or spherical polar. Then $M$ is topologically elliptic, in
particular rationally elliptic.
\end{thm}
\begin{proof}
By \tref{cover}  $(M,\G)$ is Coxeter polar, with a section $\Sigma$
of curvature $0$ or $1$. If we fix a Weyl chamber $C\subset\Sigma$,
the section $\Sigma$ is then ``tiled" by $\Pi$ translates of $C$.

 Let $p, q \in C
\subset \Sigma$ be regular points. The geodesics starting at $p$ and
ending perpendicularly
 at the orbit $\G q = \G/\H$ are exactly the critical points for the
energy function on the path space $F$ of paths starting at $p$ and
ending at $\G q$. Since the geodesic is perpendicular to the regular
orbit $\G q$, it is perpendicular to all regular orbits, in
particular to $\G p$ as well. Hence the geodesic lies in the  given
section $\Sigma$, starting at $p$ and ending at the $\Pi $ orbit of
 $q$.
  As long as $p$ is not a focal point for $\G q$,
 they are all non-degenerate and their indexes are computed
  by the Morse index theorem.  For a generic choice of $p$ and $q$ moreover,
  these geodesics  only cross the walls of the tiles  at interior points of the
  wall. Such points (together with conjugates points at multiples of $\pi$ when the section has curvature 1) are also precisely the focal points of $\G q$ along geodesic.
   At such a point the isotropy acts transitively on the normal
  sphere to the orbit type and hence the index increases by one less
  than the codimension of the orbit type.
 (When projected to the orbit space, these geodesics can be thought
    of as geodesic ``billiards").

   Since there are only finitely many walls in $C$, and since the universal cover of $\Sigma$ is $\R^n$ or $\Sph^n$
   with its canonical metric on which the lift of the polar group acts isometrically,  it
    follows that the growth of the Betti numbers ( for any field of coefficients)
   of the path space $F$ is at most polynomial.

    This path space in turn is the homotopy fiber of
the inclusion map $\G q \to M$. This suffices to complete the proof
 if the principal
    orbit $\G q = \G/\H$ is simply connected.
    To see this, consider the fibration $F \to \G/\H \to M$ and the induced sequence of iterated loop spaces

$$ \Omega F \to \Omega (\G/\H) \to \Omega M \to F
\to \G/\H \to M$$ where also $ \Omega (\G/\H) \to \Omega M \to F$ is
a fibration (cf. e.g. \cite{Ha}). Since $\pi_1(G/H)=0$, we can write
$G/H=G'/H'$ with $G'$ compact and simply connected and $H'$
connected. Using the energy function for the biinvariant metric on
$G'$, we see that $\Omega G'$ (which is now connected) is
topologically elliptic and from the spectral sequence of the
fibration $\Omega\G'\to \Omega(\G'/\H')\to \H'$ it follows that the
same is true for $\Omega(\G/\H)$. The fibration $ \Omega (\G/\H) \to
\Omega M \to F$ now shows that $\Omega M$  is topologically elliptic
as well.

Basically the same reasoning can be applied if $\G/\H$ has finite fundamental group,
say of order $\ell$: In this case one replaces $\G/H$ by its universal $\ell$ fold cover $\G'/\H'$
and the inclusion map $\G/\H \to M$ by the composed map  $\G'/\H' \to \G/\H \to M$. For each critical
point in $F$ one gets $\ell$ corresponding critical points in the homotopy fiber $F'$ of $\G'/\H' \to \G/\H \to M$
each with the same index. In particular our argument above also shows that the Betti numbers of $F'$ grow at most
polynomially, and we are done.

If $\G/\H$ has infinite fundamental group we replace $\G$ by $\SU(n)$ for $n$ large enough so that
$\G \subset \SU(n)$, and $M$ by $M' := \SU(n) \times_{\G} M$ the total space of the bundle with fiber
 $M$ associated with the principal $\G$ bundle $\SU(n) \to \SU(n)/\G$. Since $\SU(n)$ and $M$ are simply
 connected so is  $M'$. From the fibration $\G\to\SU(n)\times M\to M'$ it is clear that $M$
 is topologically elliptic if $M'$ is. Clearly $\SU(n)$ acts isometrically on $M'$ when $M'$ is
 equipped with the quotient metric induced from the product metric on $\SU(n) \times M$ where the metric
 on $\SU(n)$ is biinvariant. The orbit space $M'/\SU(n)$ is isometric to $M/\G$, and the principal
  $\SU(n)$ orbits are $\SU(n)/\H$, hence have finite fundamental group. Moreover the $\SU(n)$ action on $M'$
  is  polar with section $\Sigma$ in one of the fibres of $M\to M'\to \SU(n)/\G$.
  The proof provided
  above then shows that indeed $M'$ is topologically elliptic.
 \end{proof}

 \begin{rem*}  In dimension at most $5$, a simply connected compact manifold
 is topologically
 elliptic if and only if it is diffeomorphic to one of the known
 manifolds with non-negative curvature, i.e. $\Sph^n, n\le 5,\ \; \pm\CP^2,\; \Sph^2\times\Sph^2$
 \;
 $\CP^2\#\pm\CP^2$, \; $\Sph^3\times\Sph^2, \; \SU(3)/\SO(3)$ or the non-trivial
$\Sph^3$ bundle over $\Sph^2$. (see \cite{PP} for $n=5$). Thus in
these dimensions they are the only  manifolds that can admit a polar
action with  two dimensional section of non-negative curvature
(since in this case we can assume it also admits a polar metric with
section of constant curvature $0$ or $1$). It turns out that these
are also the only manifolds that admit cohomogeneity one actions
\cite{Ho}.
 \end{rem*}

In view of the above theorem it is natural to

\begin{conjecture}
A simply connected hyperbolic polar manifold is rationally
hyperbolic.
\end{conjecture}

This can be verified for many such manifolds using the ideas in the
proof of \tref{elliptic}. For example, recall that $M=\Sph^n \times
\Sph^n \# \Sph^n \times \Sph^n$ supports a polar action by
$\SO(n)\times\SO(n)$ with section
 $\Sph^1\times\Sph^1 \#
 \Sph^1\times\Sph^1$ of constant negative curvature.
The codimension of any stratum (neither principal nor fixed points)
is $n$. Thus the index of the geodesics are all multiples of $n-1$
and if  $ n\ge 3$  the lacunary principle shows that the energy
function on the loop space is a perfect Morse function. This easily
implies that the Betti numbers grow exponentially. The same argument
applies to arbitrarily many connected sums of $\Sph^n \times \Sph^n$
and hence provides a geometric proof of the well known fact that
such manifolds are rationally hyperbolic.

\bigskip

\providecommand{\bysame}{\leavevmode\hbox
to3em{\hrulefill}\thinspace}

\end{document}